\documentclass[a4paper,12pt]{amsart}
\usepackage[ansinew]{inputenc}
\usepackage[english]{babel}
\usepackage[T1]{fontenc}
\usepackage{csquotes}
\usepackage{amsmath,amssymb,amsthm, mathtools, bbm, ifxetex}
\usepackage{trfsigns}
\usepackage{xspace}
\usepackage{geometry}
\usepackage{enumerate}
\usepackage{hyperref}
\usepackage[foot]{amsaddr}
\usepackage{todonotes}
\usepackage{easyReview}

\numberwithin{equation}{section}
\allowdisplaybreaks

\newtheorem{satz}{Theorem}[section]
\newtheorem{proposition}[satz]{Proposition}
\newtheorem{lemma}[satz]{Lemma}
\newtheorem{korollar}[satz]{Corollary}

\newtheorem*{question*}{Question}
\newtheorem*{questions*}{Questions}

\newtheorem*{bemerkung*}{Remark}

\theoremstyle{remark}

\newtheorem*{beispiel*}{Example}

\renewcommand{\epsilon}{\varepsilon}

\def\E{{\mathbb E}}
\def\C{{\mathbb C}}
\def\P{{\mathbb P}}
\def\R{{\mathbb R}}

\def\N{{\mathbb N}}

\def\S{{\mathbb S}}

\def\Var{{\rm Var}}

\def\vol{{\rm vol}}

\def\sign{{\rm sign}}
\def\tr{{\rm tr}}
\def\LSq{{\rm LS}_q}

\def\be{\begin{equation}}
\def\en{\end{equation}}
\def\bee{\begin{eqnarray*}}
\def\ene{\end{eqnarray*}}

\title[Concentration on spheres]{Concentration of measure on spheres and\\ related manifolds}

\author{Friedrich G\"{o}tze}
\address{Friedrich G\"otze, Faculty of Mathematics, Bielefeld University, Germany}
\email{goetze@math.uni-bielefeld.de}
\author{Holger Sambale}
\address{Holger Sambale, Faculty of Mathematics, Bielefeld University, Germany}
\email{hsambale@math.uni-bielefeld.de}

\begin{document}

\begin{abstract}
    We study various generalizations of concentration of measure on the unit sphere, in particular by means of log-Sobolev inequalities. First, we show Sudakov-type concentration results and local semicircular laws for weighted random matrices. A further branch addresses higher order concentration (i.\,e., concentration for non-Lipschitz functions which have bounded derivatives of higher order) for $\ell_p^n$-spheres. This is based on a type of generalized log-Sobolev inqualities referred to as $\mathrm{LS}_q$-inequalities. More generally, we prove higher order concentration bounds for probability measures on $\R^n$ which satisfy an $\mathrm{LS}_q$-inequality. Finally, we derive concentration bounds for sequences of smooth symmetric functions on the Euclidean sphere which are closely related to Edgeworth-type expansions.
\end{abstract}

\subjclass{60B20, 60E15, 60F10}
\keywords{Concentration of measure, $\ell_p^n$-sphere, local semicircle law, logarithmic Sobolev inequality, random matrices, Sudakov's typical distributions, symmetric functions}
\thanks{Research  supported by GRF--SFB 1283/2 2021 -- 317210226}.

\date{\today}

\maketitle

\section{Introduction}

\subsection{Concentration on the sphere}
One of the classical objects in the theory of concentration of measure, chosen as initial example in the seminal monograph by M.\ Ledoux \cite{Led01}, is the Euclidean unit sphere $\S_2^{n-1} := \{x \in \R^n \colon |x|_2 = 1\}$, equipped with the uniform distribution $\nu_{2,n}$ (which may be realized as normalization of an $n$-dimensional standard Gaussian measure) . For instance, any function $f \colon \S_2^{n-1} \to \R$ which is $1$-Lipschitz and has $\nu_{2,n}$-mean zero satisfies the concentration inequality
\begin{equation}\label{LipschitzSph}
\nu_{2,n}(|f| \ge t) \le 2e^{-(n-1)t^2/2}
\end{equation}
for any $t \ge 0$, cf.\ e.\,g.\ \cite[Prop.\ 10.3.1]{BCG23}. In particular, this inequality shows that Lipschitz functions on $\S_2^{n-1}$ do not only have sub-Gaussian tails (as known from many other classical situations), but they are even stochastically of order $1/\sqrt{n}$.

The array of techniques available to study various aspects of the concentration of measure phenomenon (not only on $\S_2^{n-1}$) is vast, and we only refer to selected monographs like \cite{Led01,BLM13,BCG23}. In this note, we shall especially focus on the entropy method, which is based on so-called logarithmic Sobolev inequalities. By a classical result by Mueller and Weissler \cite{MW82}, the uniform distribution $\nu_{2,n}$ on the Euclidean unit sphere $\S_2^{n-1}$ satisfies a logarithmic Sobolev inequality with (optimal) constant $1/(n-1)$, i.\,e.,
\begin{equation}\label{LSISphere}
\mathrm{Ent}_{\nu_{2,n}} (f^2) \le \frac{2}{n-1} \int_{\S_2^{n-1}} \lvert \nabla_S f \rvert_2^2 d\nu_{2,n}
\end{equation}
for any $f \colon \S^{n-1} \to \R$ sufficiently smooth (say, locally Lipschitz). Here, $\nabla_S f$ denotes the spherical (or intrinsic) gradient of $f$, and $\mathrm{Ent}$ denotes the entropy functional (cf.\ Section \ref{LSqRn}). For the moment, we just mention that one way of deducing \eqref{LipschitzSph} is due to \eqref{LSISphere} and the so-called Herbst argument. In particular, for differentiable functions the Lipschitz condition in \eqref{LipschitzSph} may be reformulated as $|\nabla_S f| \le 1$.

Let us briefly sketch a number of further aspects of (spherical) concentration which we shall be interested in, in particular the study of so-called weighted sums $\langle X, \theta \rangle = \theta_1 X_1 + \ldots + \theta_n X_n$. In the simplest case, the $X_i$ are i.i.d.\ real random variables, while $\theta = (\theta_1, \ldots, \theta_n)$ is chosen uniformly at random from $\S_2^{n-1}$ (i.\,e., according to $\nu_{2,n}$). In this situation, Sudakov \cite{Sud78} observed that for $n$ large, the distribution functions of the $\langle X, \theta \rangle$ are concentrated about a ``typical distribution'', which can be taken as the spherical mean of the distribution function of $\langle X, \theta \rangle$. In particular removing independence and considering isotropic random vectors with suitable correlation conditions, such results have been intensely studied in the past years, cf.\ the monograph \cite{BCG23}.

One of the tools which has found application in the context of weighted sums is so-called higher order concentration of measure. By the latter, one refers to concentration bounds for functions which are non-Lipschitz (or Lipschitz of an order which leads to asymptotically non-optimal bounds) but have bounded derivatives of some higher order, cf.\ e.\,g.\ \cite{AW15}. Typical examples are quadratic forms (for order two) and, more in general, forms of some higher order. In \cite{BCG17}, second order concentration results on $\S_2^{n-1}$ have been deduced, leading to sharpened concentration bounds for functions with a bounded Hessian. These bounds have been partially extended to higher orders in \cite{BGS19}. Technically, results of this kind often build upon logarithmic Sobolev inequalities.

\subsection{Generalizations and overview}
There are multiple ways of generalizing spherical concentration results. In this note, we shall not be interested in discrete situations like $\{\pm1\}^n$ (equipped with various types of probability measures), but in spaces with a manifold structure. We address a number of situations, all of which are in some sense related to the topics introduced at the beginning of this note.

One possible generalization are Haar measures on classical compact matrix groups like the special orthogonal group $\mathrm{SO}(n)$ or the unitary group $\mathrm{U}(n)$, for which log-Sobolev results are available in \cite{Mec19}, and (in some sense related) on Stiefel and Grassmann manifolds, for which higher order concentration was studied in \cite{GS23}. In turns out that some of these results have applications in the study of Sudakov-type results for weighted random matrices, i.\,e., we address Sudakov-type questions in the context of matrices and their resolvents rather than vectors.

Moreover, we consider spheres with respect to different norms, especially $\ell_n^p$-norms, which are equipped with two classical probability measures (the cone and the surface measure). Here, a variety of concentration results is known (of which we will give a brief overview later), but it seems that the entropy method and especially higher order concentration have not yet gained much attention. In fact, in this context it is more adequate to work with a certain generalization of log-Sobolev inequalities (called $\mathrm{LS}_q$-inequalities), which is a topic of independent interest.

A further branch of results concerns the asymptotic behaviour of smooth symmetric functions on the sphere. As shown in \cite{GNU17}, such functions can be asymptotically expressed in terms of certain elementary polynomials, which under suitable conditions gives rise to Edgeworth-type expansions. In fact, this scheme can even be regarded as a universal structure for all Edgeworth-type expansions. In this context, one may ask for concentration results.

Let us provide a brief overview on the structure and the main results of this note.
\begin{itemize}
    \item In Section \ref{sec:Sud} \emph{Sudakov-type results and local laws for weighted random matrices}, we prove a Sudakov-type concentration result for weighted random matrices (more precisely, their resolvents). This is based on concentration on the special orthogonal group. Moreover, we obtain local semicircular laws for this class of weighted random matrices.
    \item Section \ref{LSqRn} \emph{Logarithmic Sobolev inequalities and generalizations} addresses a general framework of $\mathrm{LS}_q$-inequalities, mostly on $\R^n$. Our main result as stated in Theorem \ref{HOLqRn} are higher order concentration bounds in presence of $\mathrm{LS}_q$-inequalities. Examples include Hanson--Wright type inequalities.
    \item In Section \ref{LSISphSec} \emph{Log-Sobolev-type inequalities and concentration on $\ell_p^n$-spheres}, we prove $\mathrm{LS}_q$-inequalities for the cone measure $\mu_{p,n}$ on $\ell_{p,n}$-spheres, cf.\ Theorem \ref{LSICone}. This leads to first and higher order concentration results. Here, we address both the cone and the surface measure in results like Proposition \ref{LipschUnif} and Theorem \ref{HOSph}.
    \item In Section \ref{ExSect} \emph{Concentration bounds for symmetric functions on the sphere}, we derive concentration bounds for families of smooth functions on the Euclidean sphere related to Edgeworth-type expansions as studied in \cite{GNU17}. The inequalities we apply can be derived from the results in Section \ref{LSISphSec}.
\end{itemize}
Some technical tools and background information are deferred to the appendix.

In large parts of Sections \ref{LSqRn}--\ref{ExSect}, two parameters $p$ (for instance, related to $\ell_p^n$-norms) and $q$ are present, and whenever both of them appear in the same context, we assume them to be related to each other as follows.

\textbf{General convention.} Throughout this note, we regard $q \in [1,2]$ as the H\"{o}lder conjugate of $p \in [2, \infty]$ and vice versa, i.\,e., $q=p/(p-1)$ and $p=q/(q-1)$.

Finally, by $c>0$ or $C>0$ we typically denote absolute constants, while the notation $c_Q>0$ or $C_Q >0$ for some set of quantities $Q$ refers to constants which only depend on the quantities $Q$. Note also that these constants may vary from line to line if not indicated otherwise.

\section{Sudakov-type results and local laws for weighted random matrices}\label{sec:Sud}

\subsection{General setting and results}

Let $X = (X_1, \ldots,X_n)$ be a random vector in $\R^n$ and $\theta \in \S_2^{n-1}$. Starting with the work of V.\,N.\ Sudakov \cite{Sud78}, the distribution and concentration properties of the weighted sums
\[
\langle X, \theta \rangle = \sum_{j=1}^n X_j \theta_j
\]
have been subject to intense investigation, in particular with regard to certain ``typical distributions''.

Our aim is to study Sudakov-type results for weighted random matrices. More precisely, we examine the spectral distribution of weighted symmetric $n\times n$ matrices of the form $X_\Theta:= \Theta \circ X$, where $X \in \R^{n \times n}$ is a symmetric random matrix, $\Theta \in \R^{n \times n}$ is symmetric and $\circ$ denotes the Hadamard product, i.\,e.,
\[ (\Theta \circ X )_{ij} := \Theta_{ij} X_{ij}, \quad  i,j=1, \ldots n. \]
Here we shall consider the following random matrix model:
\begin{enumerate}
\item $X_{ij}$, $1 \le i \le j \le n$, are i.i.d.\ random variables,
\item $\E X_{12} = 0,$ $\E X_{12}^2 =1$.
\end{enumerate}
Note that the variance of $(X_\Theta)_{ij} $ is given by $(\Theta_{ij}^2)_{ij}$.

Let us establish conditions on the weight matrices $\Theta$ which ensure that the eigenvalue distribution of the $X_\Theta$ converges to a limit distribution as $n \to \infty$. In the simplest case, if all $\Theta_{ij} = 1/\sqrt{n}$, Wigner's semicircle law guarantees convergence towards the semicircle distribution, i.\,e., the probability measure with density
\[
p(x):=\frac 1 {2 \pi} \sqrt{4 -x^2} \mathbbm{1}_{\{|x| \le 2\}}.
\]

More generally, denote by $F_\Theta (x):= \frac 1 n \sum_{j=1}^n \mathbbm{1}_{\{\lambda_j(X_\Theta) \le x\}}$ the counting function of the eigenvalues of $X_\Theta$. If $Z$ is a symmetric matrix with entries $Z_{ij}, i\le j$, which are independent or satisfy a martingale-type dependence and have variance $\mathrm{Var}(Z_{ij}) = \sigma_{ij}^2$, let $B_i^2 := n^{-1} \sum_{j=1}^n \sigma_{ij}^2$ be the mean variance of the $i$-th row. Assuming $n^{-1} \sum_{i=1}^n |B_i^2-1| \to 0$ as $n \to \infty$ and $\max_i B_i \le C$ for some absolute constant $C$, by \cite[Th.\ 1 \& Cor.\ 1]{GNT15} we have 
\[ \Delta_Z:=\sup_x| \E_ZF_Z(x) - G(x)| \to  0\qquad \text{as } n \to \infty,\]
where $G(x)$ denotes the distribution function of the semi-circle law, if a Lindeberg condition for the $Z_{ij}$ holds. In our case, the conditions on the variance profile are satisfied if $\sigma_{ij}^2 = \Theta_{ij}^2$ is doubly stochastic, i.\,e.,
\begin{enumerate}\setcounter{enumi}{2}
    \item $\Theta \circ \Theta = (\Theta_{ij}^2)_{ij}$ is doubly stochastic.
\end{enumerate}
We also call $\Theta \circ \Theta$ the Hadamard square of $\Theta$. Moreover, the Lindeberg condition for $Z_{ij} = X_{ij} \Theta_{ij}$ is implied by
\begin{enumerate}\setcounter{enumi}{3}
    \item $\frac  1 n \sum_{j=1}^n \max_{j} |\Theta_{ji}| \to 0$ as $n \to \infty$.
\end{enumerate}
Hence, the conditions (1)--(4) are required in the sequel.

Recall that in random matrix theory, one commonly studies Stieltjes transforms of the measures under consideration. If $\mu$ is a probability measure on $\R$, its Stieltjes transform is given by
\[
s_\mu(z) = \int \frac{1}{x-z}\mu(dx)
\]
for any $z \in \C \setminus \R$. By Stieltjes inversion, one may recover the measure $\mu$ from its Stieltes transform. In particular, the Stieltjes transforms of the expected empirical distribution $\E_X \frac 1 n \sum_{j=1}^n \delta_{\lambda_k}$ of the eigenvalues $\lambda_k$ of $n^{-1/2} X$ and $X_\Theta$, respectively, are given by $\E_X s_n(z)$ and $\E_X s_\Theta (z)$, where $\E_X$ denotes taking the expectation with respect to $X$ and
\[
s_n(z):= \frac 1 n \tr ((n^{-1/2} X-z I)^{-1}) , \quad s_\Theta(z):= \frac 1 n \tr( (\Theta \circ X- z I )^{-1}).
\]
Here and in the sequel, $I$ denotes $n \times n$ identity matrix. Moreover, we shall always denote the Stieltjes transform of the semicircle law by
\[
s(z) = \frac{1}{2\pi}\int_{-2}^2 \frac{\sqrt{4-x^2}}{x-z} dx.
\]

The results in \cite{GNT15} in particular involve the convergence of the expected Stieltjes transforms $\E_X s_\Theta(z) \to s(z)$ as $n \to \infty$. Moreover, for independent $X_{ij}$ with doubly stochastic variance profile rates of convergence for $\Delta_\Theta :=\sup_x| \E_XF_\Theta(x) - G(x)| \to 0$ and $\E_X s_\Theta(z) \to s(z)$ as $n \to \infty$ were derived together with local laws in \cite{EYY12} under stronger moment conditions and more restrictive conditions on the weights, requiring $n\Theta_{ij}^2$ to be bounded above and below by positive constants. Our model of weights will be more general, only assuming (4).

We would like to study the concentration of measure phenomenon for $\Theta \mapsto \E_X s_\Theta$ in the spirit of Sudakov for our model. Since the the set of doubly stochastic square weights $P_{ij}=\Theta_{ij}^2$ is a convex set, \cite{CDS10} modelled it via a uniform distribution on it. However, since we intend to use log-Sobolev techniques, the latter turns out to be inconvenient. We therefore suggest the following probability model, where $\Theta$ is a functional on a space better suited to our purpose.

Let $\mathrm{SO}(n)$ denote the special orthogonal group of dimension $n$, which we endow with its Haar measure $\mu_{\mathrm{SO}(n)}$, and consider the mapping
\begin{equation}
O \mapsto \Theta = (\Theta_{ij})_{ij},\qquad \text{where } \Theta_{ij} := \sqrt{ \frac 1 2 (O_{ij}^2 + O_{ji}^2)}. \label{weightdef}
\end{equation}
Obviously, $\Theta$ is symmetric, and its Hadamard square is doubly stochastic by the orthogonality of $O$. In fact, it seems reasonable to conjecture that \eqref{weightdef} defines a surjective mapping into the space of the symmetric matrices with non-negative entries and doubly stochastic Hadamard squares. Therefore, we equip the space of the weight matrices with the measure induced by \eqref{weightdef}. In particular, (4) is satisfied automatically for this model (cf.\ e.\,g.\ Lemma \ref{lem:TailsMaxWei} below).

\subsubsection{Concentration results}
In studying a local semicircular law, we first establish sub-Gaussian concentration of $\E_X s_\Theta(z)$ around the expected Stieltjes transform $\E s_\Theta(z)$. The key observation is to identify $\E_X s_\Theta(z)$ as a Lipschitz function of $O \in \mathrm{SO}(n)$.

Before stating the result, let us fix some notation which we will use all over this section. As stated before, $\E_X$ denotes taking the expectation with respect to the matrix $X$, and $\E_{\mathrm{SO}(n)}$ denotes taking the expectation with respect to $O$ which is distributed according to the Haar probability measure on $\mathrm{SO}(n)$. In particular, we require $X$ and $O$ to be independent, so that $\E = \E_X\E_{\mathrm{SO}(n)}$. Moreover, given any $z \in \C$, we set $v := \Im(z)$.

\begin{proposition}\label{propCoMRM}
    Let $n \ge 3$ and $z \in \C \setminus \R$. Then, it holds that
    \[
    \mu_{\mathrm{SO}(n)}\big(\big|\E_Xs_\Theta(z) - \E s_\Theta(z)\big| \ge t\big) \le 2 \exp\Big\{- \frac{|v|^4 n^2}{768} t^2\Big\}
    \]
    for any $t \ge 0$. Furthermore, if $|X_{ij}| \le K$ for any $i,j$ and some $K > 0$, it holds that
    \[
    \mu_{\mathrm{SO}(n)}\big(\big|s_\Theta(z) - \E_{\mathrm{SO}(n)} s_\Theta(z)\big| \ge t\big) \le 2 \exp\Big\{- \frac{|v|^4 n^2}{768K^2} t^2\Big\}
    \]
    for any $t \ge 0$.
\end{proposition}

Hence, $\E_X s_\Theta(z)$ concentrates around $\E s_\Theta(z)$ for any fixed $z \in \C \setminus \R$, and the difference is stochastically of order $1/(n|v|^2)$. The concentration bounds shown in Proposition \ref{propCoMRM} have a sort of ``dual'' analogue with reversed roles of $X$ and the weights $\Theta$. For that, we need to impose further assumptions on the matrix $X$.

We say that a random vector $X \in \R^n$ admits sub-Gaussian concentration for Lipschitz functions with constant $r > 0$ if any function $f \colon \R^n \to \R$ which is $L$-Lipschitz satisfies $\E|f(X)| < \infty$ and
\begin{equation}\label{eq:LSICon}
    \P(|f(X) - \E f(X)| \ge t) \le 2 \exp\Big\{- \frac{t^2}{r^2L^2}\Big\}
\end{equation}
for any $t \ge 0$. This (standard) definition extends to random matrices by replacing $\R^n$ by $\R^{n \times n}$, for instance.

\begin{proposition}\label{prop:CoMRMX}
    Let $n \ge 3$ and $z \in \C \setminus \R$, and assume that $X$ admits sub-Gaussian concentration for Lipschitz functions with constant $r> 0$. Then, we have
    \[
    \P\big(\big|\E_{\mathrm{SO}(n)}s_\Theta(z) - \E s_\Theta(z)\big| \ge t\big) \le 2 \exp\Big\{- \frac{|v|^4 n^2}{4r^2}t^2\Big\}
    \]
    for any $t \ge 0$. Furthermore, it holds that
    \begin{align*}
    \P_X\big(\big|s_\Theta(z) - \E_X s_\Theta(z)\big| \ge t\big) &\le 2 \exp\Big\{- \frac{|v|^4 n}{4r^2\max_{ij} \Theta_{ij}^2}t^2\Big\}\\
    &\le 2 \exp\Big\{- \frac{|v|^4 n}{4r^2}t^2\Big\}
    \end{align*}
    for any $t \ge 0$.
\end{proposition}

If $\Theta_{ij} = 1/\sqrt{n}$ for all $i,j$, the second inequality of Proposition \ref{prop:CoMRMX} yields a tail bound of $2\exp(-|v|^4n^2t^2/(4r^2))$, which corresponds to results known from the literature, cf.\ e.\,g.\ \cite[Ch.\ 2.3 \& 2.4]{AGZ10}. For non-uniform weights, one arrives at weaker bounds in general, though in view of \eqref{eq:MomMaxTh} below for ``typical'' values of $\Theta$ we obtain almost the same inequality up to a logarithmic factor, and even the trivial upper bound $\Theta_{ij} \le 1$ does still lead to a meaningful result as indicated by the last estimate. Note though that in Proposition \ref{prop:CoMRMX}, $\E_X s_\Theta(z)$ will not be a typical distribution in the sense of Sudakov.

\subsubsection{A local semicircular law}
In view of Propositions \ref{propCoMRM} and \ref{prop:CoMRMX}, we would now like to study the asymptotic behaviour of the mean $\E s_\Theta(z)$. As our next result shows, it turns out to be close to the Stieltjes transform $s(z)$ of the semicircular law. This is a result of independent interest, since together with Propositions \ref{propCoMRM} and \ref{prop:CoMRMX} it gives rise to so-called local semicircular laws. By this, we mean that the empirical Stieltjes transform $s_\Theta(z)$ is close to the $s(z)$ as $n \to \infty$.

\begin{satz}\label{MainThRM}
    Let $n \ge 4$, $|v| \le 1$, and assume that $m_3:= \E |X_{12}|^3 < \infty$. Then, there is some absolute constant $c>0$ such that
    \[
    |\E s_\Theta(z) - s(z)| \le c \frac{m_3\log n}{|v|^4\sqrt{n}}.
    \]
\end{satz}

In particular, note that Theorem \ref{MainThRM} does not follow from the results in \cite{EYY12}. As mentioned before, they require boundedness assumptions on the weights $\Theta_{ij}$ from above and below which cannot be satisfied by a randomly chosen $\Theta$ as above.

The further structure of this section is as follows: in Section \ref{sec:MethPf} we introduce a number of tools and methods. The proofs of Proposition \ref{propCoMRM}, Proposition \ref{prop:CoMRMX} and Theorem \ref{MainThRM} are then given in Section \ref{sec:PfsRM}.

\subsection{Methods of the proof}\label{sec:MethPf}
Let us briefly introduce a number of concepts and auxiliary results we will frequently use in the proofs of our main results.

\subsubsection{Resolvent estimates}
In many regards, proving results for the Stieltjes transforms $s_\Theta (z)$ amounts to controlling the resolvents of the respective matrices. In general, for any matrix for $M \in \R^{n \times n}$ its resolvent is given by
\begin{equation}\label{Resolvent}
R = R(M)= R(M)(z) = R(z) := (M-zI)^{-1},
\end{equation}
where $z \in \C \setminus \R$. If $M=\Theta \circ X$, we also write $R_\Theta := (\Theta \circ X -z I)^{-1}$, so that $s_\Theta = \frac{1}{n} \tr R_\Theta$.

The following technical lemma collects a number of relations for resolvents which we shall need all over the proofs.

\begin{lemma}\label{lem:Res}
    Let $M \in \R^{n \times n}$ be a symmetric matrix with resolvent $R=R(z)$.
    \begin{enumerate}
        \item For any $1 \le i,j \le n$ and any $z \in \C \setminus \R$, it holds that
        \[
        |R_{ij}(z)| \le |v|^{-1}.
        \]
        \item Setting $\bar{R}(z) := (M - \bar{z}I)^{-1}$ for $z \in \C \setminus \R$, for any $i = 1, \ldots, n$ we have
        \[
        \sum_{j=1}^n R_{ji}^2 = (R^2)_{ii}, \quad \sum_{j=1}^n |R_{ji}|^2 = (R\bar {R})_{ii}.
        \]
        \item For any $1 \le i,j \le n$, we have
        \[
        \frac{\partial R}{\partial M_{ij}} = - R(E_{ij} + E_{ji})R \Big(1 - \frac{1}{2}\mathbbm{1}_{\{i=j\}}\Big),
        \]
        where $E_{ij}$ is the matrix with $(i,j)$-th entry $1$ and $0$ in all other entries.
        \item For any $1 \le i,j \le n$, we have
        \[
        \frac{\partial \tr(R)}{\partial M_{ij}} = - (R^2)_{ij} (1 +\mathbbm{1}_{\{i\ne j\}}).
        \]
    \end{enumerate}
\end{lemma}

\begin{proof}
    To see (1), note that that  for any symmetric matrix $M$ and its real eigenvalues $\lambda_\ell$ and normalized eigenvectors $e_\ell\in \R^n$ we have
    \[ R_{ij} = ((A-z I)^{-1})_{ij}= \sum_{\ell=1}^n e_{\ell j} e_{\ell k}(\lambda_\ell- z)^{-1}.\]
    From here, (1) and also (2) immediately follow.

    (3) and (4) readily follow from identities from matrix calculus as presented in Appendix \ref{sec:MatCal}. In detail, (3) is shown by applying Lemma \ref{ElemDer} (1) to the matrix $M-zI$. Similarly, to see (4) one notes that
    \[
    \frac{\partial \tr(R)}{\partial V_{ij}} = -\tr(R^2 (E_{ij} + E_{ji}))\Big(1 - \frac{1}{2}\mathbbm{1}_{\{i=j\}}\Big) = - (R^2)_{ij} (1 +\mathbbm{1}_{\{i\ne j\}}),
    \]
    which follows by combining Lemma \ref{ElemDer} (1), the chain rule \eqref{MatrixChain} and the circular invariance of the trace, while the second line is due to symmetry.
\end{proof}

\subsubsection{Log-Sobolev inequalities}

A central tool in our proofs are log-Sobolev techniques. We will discuss this topic in more generality in Section \ref{LSqRn}. For the moment, let us recall that a probability measure $\mu$ on $\R^n$ satisfies a logarithmic Sobolev inequality with constant $\sigma^2 > 0$ if for every sufficiently smooth function $f \colon \R^n \to \R$, we have
\[
\mathrm{Ent}_\mu(f^2) \le 2\sigma^2 \int |\nabla f|^2 d\mu,
\]
where $\mathrm{Ent}_\mu(f^2) := \int f^2\log(f^2) d\mu - (\int f^2 d\mu) \log(\int f^2 d\mu)$. A log-Sobolev inequality implies a Poincar\'{e}-type inequality with the same constant, i.\,e.,
\[
\mathrm{Var}_\mu(f) := \int f^2 d\mu - \int f^2 d\mu \le \sigma^2 \int |\nabla f|^2 d\mu.
\]
Moreover, $\mu$ admits sub-Gaussian concentration for Lipschitz functions in the sense of \eqref{eq:LSICon} with $r^2=2\sigma^2$ for any random vector $X$ with distribution $\mu$.

Logarithmic Sobolev inequalities can be generalized to metric measure spaces and, in particular, to manifolds as already seen in \eqref{LSISphere}. At this point we are especially interested in the special orthogonal group $\mathrm{SO}(n)$. Recall that by \cite[Th.\ 5.16]{Mec19}, its Haar measure $\mu_{\mathrm{SO}(n)}$ satisfies a logarithmic Sobolev inequality with constant $(n-2)/4$, i.\,e., for any $f \colon \mathrm{SO}(n) \to \R$ locally Lipschitz, it holds that
\begin{equation}\label{LSISO}
\mathrm{Ent}_{\mu_{\mathrm{SO}(n)}} (f^2) \le \frac{8}{n-2} \int_{\mathrm{SO}(n)} \lVert \nabla_{\mathrm{SO}(n)} f \rVert_\mathrm{HS}^2 d\mu_{\mathrm{SO}(n)}.
\end{equation}
Here, $\nabla_{\mathrm{SO}(n)} f$ denotes the intrinsic gradient on $\mathrm{SO}(n)$ equipped with the Hilbert--Schmidt norm (i.\,e., $\lVert \nabla_{\mathrm{SO}(n)} f \rVert_\mathrm{HS}$ equals the generalized modulus of the gradient \eqref{generalizedmodulus} on $\mathrm{SO}(n)$), which can be introduced similarly as $\nabla_S f$ in Section \ref{LSISphSec}. In more detail, the tangent space of $\mathrm{SO}(n)$ at the identity matrix $I$ is given by the space of the antisymmetric matrices, i.\,e.,
\begin{equation}\label{eq:projO}
T_I\mathrm{SO}(n) = \{A \in \R^{n \times n} \colon A + A^T = 0\},
\end{equation}
cf.\ \cite[Lem.\ 1.9]{Mec19}. Consequently, for an arbitrary $O \in \mathrm{SO}(n)$, we have $T_O\mathrm{SO}(n) = OT_I\mathrm{SO}(n)$. Noting that the projection $P_O$ onto $T_O\mathrm{SO}(n)$ is given by
\[
P_O(A) := \frac{1}{2} (A - O A^T O),
\]
for any function $f$ which is defined and smooth in a neighbourhood of $\mathrm{SO}(n) \subset \R^{n \times n}$ we have
\begin{equation}\label{orthprojSOn}
    \nabla_{\mathrm{SO}(n)} f(O) = P_O \nabla f(O),
\end{equation}
using the Euclidean gradient in the ambient space $\R^{n \times n}$. We will discuss relations of this type in more detail in Section \ref{sec:LSILqSubS}.

Noting that $\mathrm{SO}(n)$ is isomorphic to the Stiefel manifold $W_{n,n-1}$, similar calculations (in the context of concentration of measure) can also be found in \cite[Sect.\ 3]{GS23}, though the latter considers a different ambient space, leading to different representations. By \eqref{orthprojSOn} and since $\mathrm{SO}(n)$ is equipped with the Hilbert--Schmidt metric (in the notation of Section \ref{LSISphSec}, corresponding to the case $p=q=2$), it follows easily that
\[  \lVert\nabla_{\mathrm{SO}(n)} f(O)\rVert_{\mathrm{HS}} \le  \lVert\nabla f(O)\rVert_{\mathrm{HS}}, \]
cf.\ also \cite[Prop.\ 3.1]{GS23}. We also remark that due to connectivity reasons, there is no log-Sobolev inequality on the orthogonal group $\mathrm{O}(n)$, which is why we consider $\mathrm{SO}(n)$ instead.

Altogether, it follows that $\mu_{\mathrm{SO}(n)}$ satisfies a logarithmic Sobolev inequality with respect to the Euclidean gradient with the same constant as in \eqref{LSISO}, which we shall use in the sequel for simplicity. In fact, we will typically work with the slightly weaker but asymptotically correct constant $12/n$, valid for all $n \ge 3$.

\subsubsection{Moment and tail estimates for the weights}

Throughout our proofs, we need to estimate moments of $\Theta_{ij}$ from \eqref{weightdef}. To this end, recall that for any fixed $i,j$, $O_{ij}$ has the same distribution as a coordinate of a vector on the Euclidean sphere $\S_2^{n-1}$ under the uniform distribution on the sphere. Moreover, for any $i$, $O_i=(O_{ij}, j=1, \ldots,n)$ has uniform distribution on the sphere. By Lemma \ref{Momente} specialized to $p=2$, it hence follows that for any $k \ge 1$,
\begin{equation}\label{eq:MomTh}
    \E_{\mathrm{SO}(n)} \Theta_{ij}^k \le 2^{\max(k/2,1)} \E_{\mathrm{SO}(n)} |O_{ij}|^k \le c_k n^{-k/2}
\end{equation}
for suitable constants $c_k > 0$.

Moreover, we frequently need to control the maximal entry of the weight matrix $\Theta_{ij}$. The following lemma provides information on its tails.

\begin{lemma}\label{lem:TailsMaxWei}
    For any $n \ge 4$, we have
    \[
    \mu_{\mathrm{SO}(n)}\Big(\max_{ij} (\Theta(O))_{ij} \ge t\sqrt{\frac{\log(n)}{n}}\Big) \le \frac{8}{t\sqrt{2\pi}}n^{-t^2/8+2} \le \frac{8}{t\sqrt{2\pi}} \frac{1}{n^3},
    \]
    where the first inequality is valid for any $t \ge 0$ and the second for any $t \ge \sqrt{40}$.
\end{lemma}

\begin{proof}
First note that for any real $a> 0$, $(\Theta(O))_{ij} \ge a$ implies that either $|O_{ij}| > a$ or $|O_{ji}| > a$, which reduces the problem to proving that
\[
\mu_{\mathrm{SO}(n)}\Big(\max_{ij} |O_{ij}| \ge t\sqrt{\frac{\log(n)}{n}}\Big)
\]
is bounded as claimed in the lemma. Noting again that each coordinate $O_{ij}$ has the same distribution as a coordinate of a vector on the Euclidean sphere $\S_2^{n-1}$ under the uniform distribution on the sphere, we may therefore proceed similarly as in \cite[Lem.\ 2.6]{Neu23} and \cite[Ch.\ 11.1]{BCG23}.

In detail, recall that $\sqrt{n} O_{ij}$ has a symmetric density given by
\[
f_n(x) := c_n'\Big(1-\frac{x^2}{n}\Big)_+^{\frac{n-3}{2}}
\]
for some normalization constant $c_n' < 1/\sqrt{2\pi}$ for any $n \ge 2$. Hence, by \cite[Lem.\ 11.1.1]{BCG23}, we have
\[
f_n(x) \le \frac{1}{\sqrt{2\pi}} e^{-x^2/8}
\]
for all $x \in \R$ and every $n \ge 4$. Therefore, together with the usual estimate for the tails of the standard normal distribution, we arrive at
\[
\mu_{\mathrm{SO}(n)}\Big(|O_{ij}| \ge t\sqrt{\frac{\log(n)}{n}}\Big) \le 2 \int_{t\sqrt{\log(n)}}^\infty \frac{e^{-x^2/8}}{\sqrt{2\pi}} dx \le \frac{8}{t\sqrt{2\pi}}n^{-t^2/8}.
\]
By a union bound over all $n^2$ coordinates, it follows that
\[
\mu_{\mathrm{SO}(n)}\Big(\max_{ij} |O_{ij}| \ge t\sqrt{\frac{\log(n)}{n}}\Big) \le \frac{8}{t\sqrt{2\pi}}n^{-t^2/8+2} \le \frac{8}{t\sqrt{2\pi}}n^{-3}
\]
if we require $t \ge \sqrt{40}$ in the last step.
\end{proof}

In particular, it follows from Lemma \ref{lem:TailsMaxWei} that for any $k \ge 1$,
\begin{equation}\label{eq:MomMaxTh}
    \E_{\mathrm{SO}(n)}\max_{ij} \Theta_{ij}^k \le c_k \Big(\frac{\log(n)}{n}\Big)^{k/2},
\end{equation}
where $c_k>0$ are suitably chosen constants.

\subsection{Proofs}\label{sec:PfsRM}
Finally, let us proof Propositions \ref{propCoMRM}, \ref{prop:CoMRMX} and Theorem \ref{MainThRM}. The proof of Proposition \ref{propCoMRM} is easily accomplished by standard concentration arguments.

\begin{proof}[Proof of Proposition \ref{propCoMRM}]
Writing $f(V) := f_z(V) := \E_X s_V(z)$ for any symmetric $V \in \R^{n \times n}$, the proof essentially reduces to controlling the derivatives of $f(\Theta(O))$. To this end, first note that $f(V)$ has partial derivatives
\[
\frac{\partial f(V)}{\partial V_{ij}} = - \frac{1}{n}\E_X (R_V^2)_{ij} X_{ij} (1 +\mathbbm{1}_{\{i\ne j\}})
\]
similarly as in Lemma \ref{lem:Res} (4) (in its proof, replace the term $E_{ij} + E_{ji}$ by $(E_{ij} + E_{ji}) \circ X$). Moreover, for the function $\Theta = \Theta(O)$ we have 
\[ 
\frac{\partial \Theta(O)_{ij}}{\partial O_{ij}} = \frac{O_{ij}}{\Theta_{ij}} = \frac{\partial \Theta(O)_{ji}}{\partial O_{ij}},\]
while all other entries of $\Theta$ have $O_{ij}$-derivative $0$.

Thus we get by composition of maps and the chain rule \eqref{MatrixChain} for $V= \Theta(O)$ and $R_V=R_{\Theta}$
\[
\Big|\frac{\partial f(\Theta(O))}{\partial O_{ij}}\Big| \le \frac{4}{n} |\E_X (R_{\Theta}^2)_{ij} X_{ij} | \frac{|O_{ij}|}{\Theta_{ij}}
\le \frac{\sqrt{32}}{n} (\E_X |(R_{\Theta}^2)_{ij}|^2)^{1/2} (\E_X X_{ij}^2)^{1/2}.
\]
Here, the second step follows from Cauchy's inequality and the elementary estimate $|O_{ij}|/\Theta_{ij} \le \sqrt{2}$. By assumption, $\E_X X_{ij}^2=1$, so that we may use Lemma \ref{lem:Res} (2) to arrive at
\[
	\sum_{i,j=1}^n \Big(\frac{\partial f(\Theta(O))}{\partial O_{ij}}\Big)^2 \le \frac{32}{n^2} \sum_{i,j=1}^n \E_X |(R_{\Theta}^2)_{ij}|^2 = \frac{32}{n^2} \sum_{i=1}^n (R_{\Theta}^2\bar{R}_{\Theta}^2)_{ii} \le \frac{32}{n |v|^4},
\]
where the last step follows from Lemma \ref{lem:Res} (1). Plugging into \eqref{eq:LSICon} with $r^2=2\sigma^2$ and $\sigma^2 = 12/n$ completes the proof in the general case. If the $X_{ij}$ are bounded, the same arguments continue to hold, just invoking the bound $|X_{ij}| \le K$ instead of $\E_X X_{ij}^2=1$.
\end{proof}

The proof of Proposition \ref{prop:CoMRMX} is very similar and may be seen along the lines of the proof of Proposition \ref{propCoMRM} with only minor changes, noting that for $f(X) := \E_{\mathrm{SO}(n)} s_\Theta(z)$ we have
\[
\frac{\partial f(X)}{\partial X_{ij}} = - \frac{1}{n}\E_{\mathrm{SO}(n)} (R_\Theta^2)_{ij} \Theta_{ij} (1 +\mathbbm{1}_{\{i\ne j\}}).
\]
Based on this, one may easily derive the two inequalities stated in Proposition \ref{prop:CoMRMX}.

Note that the second inequality of Proposition \ref{prop:CoMRMX}, i.\,e., the $\P_X$-bound, can also be proven by a Hoffmann-Wielandt type argument similarly as in \cite[Th.\ 2.3.5]{AGZ10} but for non-uniform weights. One then applies the resulting concentration inequality to the function $f(x) := (x-z)^{-1}$, which is Lipschitz with constant $|v|^{-2}$. However, up to absolute constants, the corresponding result does not differ from the bound given in Proposition \ref{prop:CoMRMX}.

We now turn to the proof of Theorem \ref{MainThRM}, which consists of several steps. In particular, we need some further concentration results for functionals related to the resolvents $R_\Theta$ which can be proven similarly as in Proposition \ref{propCoMRM}.

\begin{lemma}\label{lem:CovBdsSOn}
    \begin{enumerate}
        \item Assuming all the $X_{ij}$ have standard Gaussian distribution $\mathcal{N}(0,1)$, it holds that that
        \begin{equation*}
            \mathrm{Cov}_X((R_\Theta)_{ii}, (R_\Theta)_{jj})| \le \frac{4}{|v|^4}\max_{ij}\Theta_{ij}^2.
        \end{equation*}
        \item It holds that that
        \begin{equation*}
            \Var_{\mathrm{SO}(n)} (\E_X (R_\Theta)_{ii}) \le  \frac{384}{n |v|^4}.
        \end{equation*}
    \end{enumerate}
\end{lemma}

\begin{proof}
To see (1), note that using Lemma \ref{lem:Res} (3), we have
\begin{align*}
\frac{\partial (R_\Theta)_{ii}}{\partial X_{jk}} &= -\Big(R_\Theta ((E_{jk} + E_{kj})\circ \Theta)) R_\Theta \Big(1 - \frac{1}{2}\mathbbm{1}_{\{j=k\}}\Big)\Big)_{ii} \\
&= - (R_\Theta)_{ij} \Theta_{jk} (R_\Theta)_{ki}(1 +\mathbbm{1}_{\{j\ne k\}}),
\end{align*}
where the second step is due to symmetry. Recall that the standard Gaussian measure satisfies a log-Sobolev inequality with constant $1$ and hence also a Poincar\'{e}-type inequality with the same constant. This yields
\begin{align*}
\Var_X((R_\Theta)_{ii})  &\le  \E_X \lVert\nabla_X  (R_\Theta)_{ii}\rVert_{\mathrm{HS}}^2 \le 4\E_X \sum_{j,k=1}^n |R_{ij} \Theta_{jk} R_{ki}|^2\\
&\le 4 \E_X \Big(\sum_{j=1}^n |R_{ij}|^2 \sum_{k=1}^n |R_{ki}|^2\Big)\max_{jk}\Theta_{jk}^2 \\
&\le \frac {4} {|v|^4 } \max_{jk}\Theta_{jk}^2,
\end{align*}
where the last step combines Lemma \ref{lem:Res} (1)\&(2). Thus we have
\begin{equation*}
|\mathrm{Cov}_X((R_\Theta)_{ii}, (R_\Theta)_{jj})| \le (\Var_X((R_\Theta)_{ii})\Var_X((R_\Theta)_{jj}))^{1/2} \le \frac{4}{ |v|^4}\max_{jk}\Theta_{jk}^2
\end{equation*}
as claimed in (1).

To see (2), we proceed similarly: first note that
\begin{align*}
\frac{\partial \E_X (R_V)_{ii}}{\partial V_{jk}} &= -\E_X  \Big(R_V ((E_{jk} + E_{kj})\circ X)) R_V \Big(1 - \frac{1}{2}\mathbbm{1}_{\{j=k\}}\Big)\Big)_{ii}\\
&= - \E_X (R_V)_{ij} X_{jk} (R_V)_{ki}(1 +\mathbbm{1}_{\{j\ne k\}}).
\end{align*}
By similar arguments as in the proof of Proposition \ref{propCoMRM}, it hence follows that
\begin{equation*}
\Big|\frac{\partial (R_{\Theta(O)})_{ii}}{\partial O_{jk}}\Big| \le \sqrt{32} (\E_X |(R_{\Theta})_{ij}|^2 |(R_{\Theta})_{ki}|^2)^{1/2} (\E_X X_{jk}^2)^{1/2}.
\end{equation*}
Thus by Lemma \ref{lem:Res} (2), we obtain
\begin{equation*}
	\sum_{j,k} \Big(\frac{\partial (R_{\Theta(O)})_{ii}}{\partial O_{jk}}\Big)^2 \le 32 \sum_{j,k=1}^n \E_X |(R_{\Theta})_{ij}|^2|(R_{\Theta})_{ki}|^2 = 32 \E_X |(R_{\Theta}\bar{R_{\Theta}})_{ii}|^2 \le \frac{32}{|v|^4}.
\end{equation*}
Finally, noting that by the Poincar\'{e} inequality on $\mathrm{SO}(n)$ we have
\[
\Var_{\mathrm{SO}(n)}(\E_X (R_\Theta)_{ii}) \le \frac{12}{n} \E_{\mathrm{SO}(n)} \lVert \nabla_O (\E_X (R_\Theta)_{ii}) \rVert_\mathrm{HS}^2
\]
completes the proof.
\end{proof}

Next, we show we may replace the random variables $X_{ij}$ by independent standard Gaussians $G_{ij}= G_{ji}$ in the sense that the resolvents are close to each other.

\begin{lemma}\label{lem:Gauss}
It holds that
\[ \max_{i,j}|\E  ((X\circ \Theta  -zI)^{-1})_{ij} -\E ((G\circ \Theta  -zI)^{-1})_{ij}| \le \frac{ c m_3 }{\sqrt{n} |v|^4}.
\]
\end{lemma}

\begin{proof}
The bound follows by inspection of the error terms in the proof of \cite[Th.\ 2]{GNT15}. Note that the latter also includes a truncation and recentering of $X_{ij}$ and $G_{ij}$ at $\sqrt{n}$ resulting in random variables $\hat{X}$, $\hat{G}$ which we avoid here by assuming three or more moments.

To start, let us rewrite
\[ \E  ((X\circ \Theta  -zI)^{-1})_{ij}
 -\E ((G\circ \Theta  -zI)^{-1})_{ij} = \int_0^{\pi/2} d\varphi \frac { d } {d \varphi}  \E  ((Z_{\varphi}\circ \Theta  -zI)^{-1})_{ij}, \]  
where $Z_{\varphi} := Z := \cos(\varphi) X +\sin(\varphi) G$. The plan of the proof is now to estimate the derivative by Taylor expansion using matrix derivatives with respect to $\varphi$. In detail, with $s_\varphi (z) = \frac{1}{n} \tr(R(Z))$, as in Lemma \ref{lem:Res} (4) we see that
\begin{equation}\label{eq:derivative}
    \frac{\partial s_\varphi(z)}{\partial \varphi} = -\frac{1}{n} \sum_{i,j = 1}^n \E \frac{\partial Z_{ij}}{\partial \varphi} \Theta_{ij} (R^2)_{ij}
\end{equation}

For all $1 \le i \le j \le n$, let us regard $(R^2)_{ij}$ as a function of $Z_{ij}$ (conditionally on the other variables), writing $(R^2)_{ij} = u_{ij}(Z_{ij})$. By Taylor's formula, we have
\begin{equation*}
u_{i j} (Z_{ij}) = u_{i j}(0) + Z_{ij}\frac{\partial u_{i j}}{\partial Z_{ij}}(0) +
\E_{U} U (1-U) Z_{ij}^2 \frac{\partial^2 u_{i j}}{\partial Z_{ij}^2}(UZ_{ij}),
\end{equation*}
where $U$ has a uniform distribution on $[0,1]$ and is independent of all the $Z_{ij}$. Multiplying both sides by
\[
Z_{ij}' := \frac{\partial Z_{ij}}{\partial \varphi} = - \sin (\varphi) X_{i j} + \cos (\varphi) G_{i j}
\]
and $\Theta_{ij}$ and taking expectation, we obtain
\begin{align*}
\E Z_{i j}' \Theta_{ij} u_{i j} (Z_{ij}) &=  \E Z_{i j}' \Theta_{ij} u_{i j}(0) + \E Z_{i j}'Z_{i j}\Theta_{ij}\frac{\partial u_{i j}}{\partial Z_{i j}}(0)\\
&\quad + \E U (1-U) Z_{i j}' Z_{i j}^2 \Theta_{ij} \frac{\partial^2 u_{i j}}{\partial Z_{i j}^2}(U Z_{i j}).
\end{align*}
Since $X_{ij}$ and $G_{ij}$ are independent of $\Theta_{ij}$ and $u_{ij}(0)$ and have mean zero, it follows that
\[
\E Z_{i j}' \Theta_{ij} u_{i j}(0) = 0.
\]
Hence, we can rewrite~\eqref{eq:derivative} in the following way
\begin{align*}
\frac{\partial s_\varphi(z)}{\partial \varphi} &= -\frac{1}{n} \sum_{i,j = 1}^n \E Z_{i j}'Z_{i j}\Theta_{ij}\frac{\partial u_{i j}}{\partial Z_{i j}}(0) -
\frac{1}{n} \sum_{i,j = 1}^n  \E U (1-U) Z_{i j}' Z_{i j}^2 \Theta_{ij} \frac{\partial^2 u_{i j}}{\partial Z_{i j}^2}(U Z_{i j})\\
& = -(\mathbb A_1 + \mathbb A_2).
\end{align*}

It is easy to see that
\begin{equation}\label{eq:ProdZZ'}
Z_{ij}' Z_{i j} = - \frac{1}{2} \sin (2 \varphi) X_{ij}^2 + \cos^2 (\varphi) X_{i j} G_{i j} - \sin^2 (\varphi) X_{i j} G_{i j} + \frac{1}{2} \sin (2 \varphi) G_{ij}^2.
\end{equation}
The random variables $G_{ij}$ are independent of $X_{ij}$ and both are independent of $\Theta_{ij}$ and $(\partial u_{ij})/(\partial Z_{ij})(0)$. Using this fact we conclude that
\begin{align*}
&\E X_{i j} G_{i j} \Theta_{ij} \frac{\partial u_{i j}}{\partial Z_{i j}}(0) = \E G_{ij} \E X_{i j} \Theta_{ij} \frac{\partial u_{i j}}{\partial Z_{i j}}(0)  = 0, \\
&\E G_{ij}^2 \Theta_{ij}^2 \frac{\partial u_{ij}}{\partial Z_{ij}}(0) = \E X_{ij}^2\Theta_{ij}^2\frac{\partial u_{ij}}{\partial Z_{ij}}(0),
\end{align*}
and hence, it follows that $\mathbb{A}_1 = 0$.

It remains to estimate $\mathbb A_2$. To this end, we first calculate $(\partial^2 u_{i j})/(\partial Z_{i j}^2)$. Using Lemma \ref{ElemDer} (3) for $M=(\Theta \circ Z)-zI$ and symmetry, we have
\begin{align*}
\frac{\partial u_{ij}}{\partial Z_{ij}} &= - \left [ R \frac{\partial (\Theta \circ Z) }{\partial Z_{ij}} R^2 \right ]_{ij} - \left [R^2 \frac{\partial (\Theta \circ Z) }{\partial Z_{ij}} R \right ]_{ij}\\
&=-\Theta_{ij}\Big( [R(E_{ij} + E_{ji})R^2]_{ij} + [R^2(E_{ij} + E_{ji})R]_{ij}\Big)\\
&=-\Theta_{ij}(R_{ij}(R^2)_{ij} + R_{ii}(R^2)_{jj} + (R^2)_{ij}R_{ij} + (R^2)_{ii}R_{jj}).
\end{align*}
By Leibniz rule, it is now easy to see that the same calculations together with Lemma \ref{lem:Res} (3) show that $(\partial^2 u_{i j})/(\partial Z_{i j}^2)$ is a sum of terms of the form
\[
\Theta_{ij}^2 (R^2)_{i_1i_2}R_{i_3i_4}R_{i_5i_6}
\]
with indices $i_k \in \{i,j\}$. By Lemma \ref{lem:Res} (1), this yields the bound
\[
\left|\frac{\partial^2 u_{ij}}{\partial Z_{ij}^2}(UZ_{ij}) \right| \le c \frac{\Theta_{ij}^2}{|v|^4},
\]
where $c$ is some absolute (combinatorial) constant. Multiplying \eqref{eq:ProdZZ'} by $Z_{ij}$, we therefore obtain by elementary estimates that
\begin{equation*}
|\mathbb A_2| \le   \frac 1 n \sum_{i,j=1}^n \E_{\mathrm{SO}(n)} \frac{ c \E_X  |X_{12}|^3  \Theta_{ij}^3}{|v|^4} \le \frac{c_1 \E_X  |X_{12}|^3} {n^{1/2} |v|^4},
\end{equation*}
using \eqref{eq:MomTh}. This finishes the proof.
\end{proof}

Finally, we give the proof of Theorem \ref{MainThRM}.

\begin{proof}[Proof of Theorem \ref{MainThRM}]
We first prove the theorem in the case that all the $X_{ij}$, $i\le j$, have standard Gaussian distribution $\mathcal{N}(0,1)$. Moreover, all over the proof we shortly write $R$ instead of $R_\Theta$. Recalling the resolvent identity
\begin{align*}
    (X\circ \Theta -zI)^{-1} &= - \frac 1 z (X\circ \Theta -zI)(X\circ \Theta -zI)^{-1} + \frac 1 z (X\circ \Theta) (X\circ \Theta -zI)^{-1}\\
    &= - \frac 1 z +\frac 1 z (X\circ \Theta)  (X\circ \Theta-zI)^{-1}
\end{align*}
and using that $\E_X X_{ij} =0$ and $\mathrm{Var}_X(X_{ij}) = 1$, we get for the diagonal entries
 \[
 \E R_{ii}= - \frac 1 z + \frac 1 z \sum_{j=1}^n \E X_{ij} \Theta_{ij}R_{ji}.
 \]
Gaussian integration by parts yields
\[ \E X_{ij}\Theta_{ij} R_{ji}  = \E \frac {\partial} {\partial  X_{ij}} R_{ji} \Theta_{ij},
\]
and hence it follows from Lemma \ref{lem:Res} (3) that
\begin{align*}
 \E R_{ii} &= - \frac 1 z  -\frac 1 z \E \sum_{j=1, j \ne i}^n  (R_{ji} R_{ij} + R_{ii}R_{jj})\Theta_{ij}^2   - \E \frac 1 z R_{ii}^2\Theta_{ii}^2\\
 &=- \frac 1 z  -\frac {1}{z} \E \sum_{j=1}^n  R_{ji}^2 \Theta_{ij}^2 
 - \E \frac {R_{ii}} {z}  \sum_{j=1}^n R_{jj} \Theta_{ij}^2 + \E \frac {R_{ii}^2\Theta_{ii}^2}{z}.
\end{align*}

Summing up these terms, we arrive at the identity
\begin{equation}\label{eq:IdsTh}
\E s_\Theta(z) = -  \frac 1 z  -  \frac 1 z (\E s_\Theta(z)^2 + \Delta_{n,3}+ \Delta_{n,2} +\Delta_{n,1}),
\end{equation}
where we set
\begin{align*}
\Delta_{n,1} &:=  \frac 1 z \E n^{-1} \sum_{i=1} R_{ii}^2\Theta_{ii}^2, \qquad \Delta_{n,2} := -\frac  1 {z} \E n^{-1}\sum_{j=1,i=1}^n    R_{ji}^2 \Theta_{ij}^2,\\
\Delta_{n,3} &:= \frac 1 z \big( (\E s_n)^2 -  \frac 1 { n} \E \sum_{i,j} R_{ii} R_{jj} \Theta_{ij}^2\big). 
\end{align*}

By Lemma \ref{lem:Res} (1), we have
\begin{align*}
|\Delta_{n,1} |  & \le \frac {1} { n|z| } \frac{1}{|v|^2}  \E \sum_{i=1}^n \Theta_{ii}^2 = \frac {1} { n|z| |v|^2 },\\
|\Delta_{n,2}|  & \le  \frac {1} { n |z| } \E \sum_{i,j=1}^n |R_{ij}|^2 \Theta_{ij}^2 \le  \frac{1} {|z| } \Big(\E \sum_{i,j=1}^n \frac{|R_{ij}|^4}{n} \Big )^{1/2} \Big(\E \sum_{i,j=1}^n \frac{\Theta_{ij}^4}{n}\Big)^{1/2}\\
&\le \frac{1}{ |z|   |v|^2 \sqrt{n} }.
\end{align*}
Here we used \eqref{eq:MomTh} as well as $\sum_{i,j} |R_{ij}|^4/n \le 1/|v|^4$, which again follows from Lemma \ref{lem:Res} (1)\&(2).

The remaining bound of $\Delta_{n,3}$ requires a double concentration argument for $R_{ii}=R_{ii}(\Theta \circ X)$$ $ as a function the Gaussian matrix $X$ and of $\Theta$ in the Hadamard product $\Theta \circ X$. Setting $T_{ij}:= \E_X R_{ii} \E_X R_{jj} - \E R_{ii} \E R_{jj}$, we may rewrite
\begin{align*}
z \Delta_{n,3} &= \frac{1} {n} \sum_{i,j=1}^n \frac{1} {n} (\E R_{ii}) (\E R_{jj}) -  \frac 1 { n} \E_{\mathrm{SO}(n)} \sum_{i,j=1}^n \Theta_{ij}^2 \Big((\E_X R_{ii}) (\E_X R_{jj})\\
&\hspace{2cm}+\mathrm{Cov}_X(R_{ii},  R_{jj})\Big)\\
& =  \frac{1} {n} \E_{\mathrm{SO}(n)}\sum_{i,j=1}^n \Big( \Big(\frac{1} {n}- \Theta_{ij}^2 \Big) \E R_{ii} \E R_{jj} -  \Theta_{ij}^2 \Big(\mathrm{Cov}_X(R_{ii},  R_{jj})\\
&\hspace{2cm}- (\E_X R_{ii} \E_X R_{jj} - \E R_{ii} \E R_{jj})\Big)\Big).
\end{align*}
Hence, with $B_n:=\{ O \in \mathrm{SO}(n) \colon \max_{ij}\Theta(O)_{ij} \le \sqrt{A_n/n} \}$ for $A_n:= 40\log(n)$ we have
\begin{align*}
|z \Delta_{n,3}| &\le  \frac{1} {n} \Big|\E_{\mathrm{SO}(n)} \sum_{i,j=1}^n \Big( \Theta_{ij}^2 \mathrm{Cov}_X(R_{ii},  R_{jj})\Big) \Big|  + \frac {A_n} { n^2} \E_{\mathrm{SO}(n)}  \mathbbm{1}_{B_n}  \sum_{i,j=1} |T_{ij}|\\
&\hspace{1cm} + \E_{\mathrm{SO}(n)} \mathbbm{1}_{B_n^c} \frac 1 n \sum_{i,j=1} |T_{ij}|.  
\end{align*}

Note that by Lemma \ref{lem:CovBdsSOn} (2) and Lemma \ref{lem:Res} (1), we have
\begin{align}
\E_{\mathrm{SO}(n)} |T_{ij}| &\le \E_{\mathrm{SO}(n)} (|\E_X R_{ii} - \E R_{ii}| |\E_X R_{jj}| + |\E R_{ii}| |\E_X R_{jj} - \E R_{jj}| )\nonumber \\
 &\le \frac{2 c}{\sqrt{n} |v|^3 }.\label{tbound}
 \end{align}
Furthermore, applying Lemma \ref{lem:TailsMaxWei} yields
\[
\E_{\mathrm{SO}(n)} \mathbbm{1}_{B_n^c} \frac 1 n \sum_{i,j=1} |T_{ij}|\le  2 n^{-2} \frac{n}{|v|^2} =  \frac{2}{n |v|^2}.
\]
This shows together with Lemma \ref{lem:CovBdsSOn} (1) and \eqref{tbound} that $z \Delta_{n,3}$ may be bounded by
\begin{align*}
 |z \Delta_{n,3}| &\le \frac c { n |v|^4}\Big(\sum_{i,j=1}^n  \E_{\mathrm{SO}(n)} \Theta_{ij}^2 \max_{kl}\Theta_{kl}^2\Big) + \frac{A_n} {\sqrt{n} |v|^3 } + \frac{2}{n |v|^2} \\
 &\le \frac {c \log n} {\sqrt{n}  |v|^3} + \frac {c \log n} { n |v|^4} + \frac{2}{n |v|^2} , 
\end{align*}
using \eqref{eq:MomMaxTh}.

Thus we conclude from \eqref{eq:IdsTh} with $|\Delta_n|  \le  |z| (|\Delta_{n,1}| + |\Delta_{n,2}| + |\Delta_{n,3}|)$
\begin{equation}
(\E s_\Theta(z))^2  + z\E s_\Theta(z) +1 = \Delta_n   \label{stielt}
\end{equation} 
Moreover, recall that the Stieltjes transform of the Wigner measure $\frac {1} {2\pi} {\sqrt{ 4 -x^2}}$ satisfies
\begin{equation}
 s(z)^2 + z s(z) +1 =0 , \quad  \text{ or } s(z)= -\frac 1 2(z-\sqrt{z^2-4}),  \label{wigner} 
\end{equation}
choosing the branch with non-negative imaginary part on the upper half-plane of $z$. Assuming that $|\Delta_n| \le 1/10 $ for $n$ sufficiently large and recalling that $|v| \le 1$ one checks that the  solutions $\E s_\Theta(z)$ of the quadratic equations \eqref{stielt} and $s(z)$ of \eqref{wigner} satisfy
\begin{equation}\label{eq:ResGauss}
|\E s_\Theta(z) - s(z)|  \le   2 |\Delta_n| \le c \frac{\log n}{|v|^3\sqrt{n}}
\end{equation}
for a suitable constant $c > 0$, cf.\ e.\,g.\ the proof of \cite[Th.\ 2.3]{CG13}. This concludes the proof of Theorem \ref{MainThRM} for Gaussian matrices (note in particular the somewhat better error term in this situation). Moreover, the general case follows by combining \eqref{eq:ResGauss} with Lemma \ref{lem:Gauss}.
\end{proof}

\section{Logarithmic Sobolev inequalities and generalizations}\label{LSqRn}

Let $(\Omega,d)$ be a metric space, and let $f \colon \Omega \to \R$ be locally Lipschitz. Then, for any $x \in \Omega$ which is no isolated point,
\begin{equation}
\label{generalizedmodulus}
\Gamma(f)(x) := \limsup_{d(x,y) \to 0^+} \frac{|f(x)-f(y)|}{d(x,y)}
\end{equation}
is called the generalized modulus of the gradient. If $x\in\Omega$ is an isolated point, we formally set $\Gamma(f)(x):=0$. For instance, consider $\Omega=\R^n$ together with the Euclidean norm $|\cdot|$. In this case, if $f \colon \R^n \to \R$ is a differentiable function, $\Gamma(f) = |\nabla f|$ is the Euclidean norm of the usual gradient. Likewise, is $\R^n$ is equipped with the $\ell_p^n$-norm $|\cdot|_p$ and $f$ is a differentiable function, $\Gamma(f) = |\nabla f|_q$ is the $\ell_q^n$-norm of the usual gradient. The generalized modulus of the gradient preserves many identities from calculus in form of inequalities, such as a ``chain rule inequality''
\begin{equation}
\label{chainrule}
\Gamma (u(f)) \le |u'(f)|\Gamma(f),
\end{equation}
where $u$ is a smooth function on $\R$ (or, more generally, a locally Lipschitz function).

Let $\mu$ be a Borel probability measure on $(\Omega,d)$, and let $q \in [1,2]$. Then, we say that $\mu$ satisfies an $\mathrm{LS}_q$-inequality (with respect to $\Gamma$) with constant $\sigma^q > 0$ (in short: $\mu$ satisfies $\mathrm{LS}_q(\sigma^q)$) if
\begin{equation}
\label{LSq}
\mathrm{Ent}_{\mu}(|f|^q) \le \sigma^q \int \Gamma(f)^q d\mu
\end{equation}
for all $f \colon \Omega \to \R$ locally Lipschitz. Here, $\mathrm{Ent}_\mu(|f|^q) = \int |f|^q \log(|f|^q) d\mu - \int |f|^q d\mu \linebreak[2]\log(\int |f|^q d\mu)$ denotes the entropy. For $q=2$, we get back the usual notion of logarithmic Sobolev inequalities (up to normalization constant), of which $\mathrm{LS}_q$-inequalities may be regarded as an extension. For a broader view on $\mathrm{LS}_q$-inequalities, cf.\ e.\,g.\ the monograph \cite{BZ05}.

By the chain rule inequality \eqref{chainrule} applied to $|f|^{q/2}$, \eqref{LSq} may be rewritten as a modified logarithmic Sobolev inequality
\[
\mathrm{Ent}_\mu(f^2) \le \Big(\frac{2\sigma}{q}\Big)^q \int \Big|\frac{\Gamma (f)}{f}\Big|^q f^2 d\mu
\]
in the sense of e.\,g.\ \cite{ABW17}. Moreover, \cite[Th.\ 2.1]{BZ05} states that \eqref{LSq} implies the Poincar\'{e}-type (or $\mathrm{SG}_q$-) inequality
\begin{equation}
\label{SGq}
\int \Big| f - \int f d\mu \Big|^q d\mu \le \frac{4\sigma^q}{\log(2)} \int \Gamma(f)^q d\mu,
\end{equation}
valid for any $f \colon \Omega \to \R$ locally Lipschitz.

As for further implications, it was shown in \cite[Prop.\ 2.3 \& Cor.\ 2.5]{BZ05} that $\mathrm{SG}_q (\sigma^q)$ implies $\mathrm{SG}_2 (\sigma^2)$ (up to constant). Generalizing these arguments, one may prove that for $1 \le q \le q' \le 2$, $\mathrm{SG}_q(\sigma^q)$ implies $\mathrm{SG}_{q'}(\sigma^{q'})$ (again, up to constant) and even that $\mathrm{LS}_q(\sigma^q)$ implies $\mathrm{LS}_{q'}(\sigma^{q'})$ with respect to the same generalized modulus of gradient $\Gamma$.

It is well-known that $\mathrm{LS}_q$-inequalities yield concentration bounds for Lipschitz functions. More precisely, given \eqref{LSq}, for any function $g \colon \Omega \to \R$ such that $\Gamma(f) \le 1$, it holds that
\begin{equation}
\label{LipschitzLq}
    \mu\Big(\Big|g - \int g d\mu\Big| \ge t\Big) \le 2 \exp\Big\{- \Big(\frac{q-1}{\sigma}\Big)^p t^p\Big\}
\end{equation}
for any $t \ge 0$ by \cite[Th.\ 1.3]{BZ05}.

Let us now turn to the case where $\Omega=\R^n$. Here, it is natural to consider $\mathrm{LS}_q$-inequalities for $\R^n$ being equipped with the $\ell_p^n$-norm, hence $\Gamma(f) = |\nabla f|_q$ (we keep the latter notation even when $f$ is not differentiable). Classical examples of measures on $\R^n$ satisfying $\mathrm{LS}_q$-inequalities are the so-called $p$-generalized Gaussian distributions $\mathcal{N}_p$ for $p \ge 2$. For $n=1$, these are the probability measures with Lebesgue density
\[
f_p(x) := c_p e^{-|x|^p/p},\qquad \text{where}\qquad c_p := \frac{1}{2p^{1/p}\Gamma(1 + \frac{1}{p})},
\]
and for $n \ge 2$, $\mathcal{N}_p$ is the $n$-fold product of this distribution.

\begin{proposition}\label{LSIp-Gauss}
For any real $p \ge 2$, the $p$-generalized Gaussian distribution $\mathcal{N}_p$ satisfies an $\mathrm{LS}_q$-inequality with a constant $\sigma^q = 2^qq^{q-1}$. More precisely, for any $f \colon \R^n \to \R$ locally Lipschitz,
\[
\mathrm{Ent}_\mu(|f|^q) \le 2^qq^{q-1} \int \lvert \nabla f \rvert_q^q d\mu.
\]
\end{proposition}

\begin{proof}
For completeness, let us provide a brief proof of this well-known result. Combining Proposition 3.2 and Eq.\ (3.7) from \cite{BL00} for $V(x) = |x|_p^p/p$, an $\mathrm{LS}_q$-inequality holds with $c=(q/\gamma)^{q-1}$ for $\gamma$ such that for every $x,y \in \R^n$,
\[
V(x) + V(y) - 2V\Big(\frac{x+y}{2}\Big) \ge \frac{\gamma}{p} |x-y|_p^p.
\]
For our choice of $V(x)$, this relation is valid with optimal $\gamma = 2^{-p}$, cf.\ \cite{BL00} (below (5.2)) or \cite{LT79}. This finishes the proof.
\end{proof}

A general feature of logarithmic Sobolev-type inequalities and the entropy method is that they also give rise to moment inequalities and, as a consequence, to as higher order concentration. In what follows, we derive higher order concentration results for measures on $\R^n$ in presence of an $\mathrm{LS}_q$-inequality in the spirit of \cite[Sect.\ 1.3]{BGS19} (which essentially corresponds to the case of $p=q=2$).

To introduce some notation, we shall work with $\mathcal{C}^d$-functions $f \colon \R^n \to \R$, whose tensor of $k$-fold derivatives we denote by $f^{(k)}$. In other words, its entries
\begin{equation}
\label{Hesseallgem}
f^{(k)}_{i_1 \ldots i_k}(x) = \partial_{i_1 \ldots i_k} f(x), \qquad k = 1,2,\dots
\end{equation}
are given by the $k$-fold (continuous) partial derivatives of $f$ at $x \in \R^n$. By considering $f^{(k)}(x)$ as a symmetric multilinear $d$-form, we moreover define operator-type norms with respect to the $\ell_p^n$-norms by
\begin{equation}
\label{Operatornorm}
|f^{(d)}(x)|_{\mathrm{op}(q)} = \sup \left\{ f^{(k)}(x)[v_1, \ldots, v_k] \colon v_1, \ldots, v_k \in \S_p^{n-1}\right\}.
\end{equation}
In particular, $|f^{(1)}|_{\mathrm{op}(q)} = |\nabla f|_q$ is the $\ell_q^n$-norm of the gradient. Furthermore, we will use the short-hand notation
\begin{equation}
\label{Lpnorm}
\lVert f^{(k)} \rVert_{\mathrm{op}(q), r} =
\left(\int |f^{(d)}|_{\mathrm{op}(q)}^r\, d\mu\right)^{1/r}, \qquad r \in (0, \infty].
\end{equation}

\begin{satz}\label{HOLqRn}
    Let $\mu$ be a probability measure on $\R^n$ which satisfies an $\mathrm{LS}_q$-inequality
    \[
    \mathrm{Ent}_{\mu}(|f|^q) \le \sigma^q \int |\nabla f|_q^q d\mu,
    \]
    and let $f\colon \R^n \to \R$ be a $\mathcal{C}^d$-function with $\mu$-mean zero.
    \begin{enumerate}
        \item Assuming that $\lVert f^{(k)} \rVert_{\mathrm{op}(q),q} \le \sigma^{d-k}$ for all $k=1,\ldots,d-1$ and $\lVert f^{(d)} \rVert_{\mathrm{op}(q),\infty} \le 1$, it holds that
        \[
        \int \exp\Big\{\frac{c_{p,d}}{\sigma^p} |f|^{p/d} \Big\} d\mu \le 2,
        \]
        where a possible choice of the constant $c_{p,d}$ is given by
        \[
        c_{p,d} = \frac{(\log(2)(p^{1/(p-1)}-(p-1)^{1/(p-1)}))^{p-1}}{2^{2p-1}p^2(p-1)^{p-1}e}.
        \]
        \item For any $t\ge 0$, it holds that
        \[
        \mu(|f| \ge t)\le
        2 \exp\Big\{- \frac{C_{p,d}}{\sigma^p} \min\Big(\min_{k=1, \ldots,d-1} \Big(\frac{t}{\lVert f^{(k)} \rVert_{\mathrm{op}(q),q}}\Big)^{p/k}, \Big(\frac{t}{\lVert f^{(d)} \rVert_{\mathrm{op}(q),\infty}}\Big)^{p/d}\Big)\Big\},
        \]
    where a possible choice of the constant $C_{p,d}$ is given by
    \[
    C_{p,d} = \frac{\log(2)^p}{4^{p-1}p(p-1)^{p-1}d^pe^p}.
    \]
    \end{enumerate}
\end{satz}

In a more general setting, higher order concentration results for polynomial functions in presence of modified log-Sobolev inequalities can be found in \cite[Th.\ 3.12]{ABW17}. Note though that evaluating their condition $(G_{K,\alpha,\beta})$, one arrives at $K=1$, $\alpha=q$ and (necessarily) $\beta=2$, so that the quantity $L(K,D,\alpha,\beta)$ which appears in their bounds reads
\[
L(K,D,\alpha,\beta) = L(\sigma,q) = \frac{1}{q-1}\Big(\frac{2\sigma}{q}\Big)^{q/2} + \Big(\frac{1}{q-1} + 2^{1/q}\Big) \frac{2\sigma}{q}.
\]
If $\sigma = o(1)$ (a situation we especially consider in this note), the first term will dominate, which seems to lead to non-optimal bounds as compared to the methodology based on \eqref{LSqLp} below. 

To prove Theorem \ref{HOLqRn}, we start with a technical lemma which establishes a relation between the $q$-operator norms of the derivatives of consecutive orders. This extends \cite[Lem.\ 4.1]{BGS19}, whose proof is readily adapted.

\begin{lemma}
\label{itGradHess}
Given a $\mathcal{C}^d$-smooth function $f \colon \R^n \to \mathbb{R}$, $d \in \mathbb{N}$, for any $x \in \R^n$ it holds that
\[
|\nabla \lvert f^{(d-1)}(x) \rvert_{\mathrm{op}(q)}|_q \le 
\lvert f^{(d)}(x) \rvert_{\mathrm{op}(q)}.
\]
\end{lemma}

\begin{proof}
Indeed, for any $h \in \mathbb{R}^n$, by the triangle inequality,
\begin{align*}
 & 
\big|\,\lvert f^{(d-1)}(x+h) \rvert_{\mathrm{op}(q)} - 
\lvert f^{(d-1)}(x) \rvert_{\mathrm{op}(q)}\big|
 \, \le \,
\lvert f^{(d-1)}(x+h) - f^{(d-1)}(x) \rvert_{\mathrm{op}(q)}\\
 = \, &
\sup \{ (f^{(d-1)}(x+h) - f^{(d-1)}(x))[v_1, \ldots, v_{d-1}] 
\colon v_1, \ldots, v_{d-1} \in \S_p^{n-1} \},
\end{align*}
while, by the Taylor expansion,
$$
(f^{(d-1)}(x+h) - f^{(d-1)}(x))[v_1, \ldots, v_{d-1}] = 
f^{(d)}(x)[v_1, \ldots, v_{d-1}, h] + o(|h|_\infty)
$$
as $h \to 0$. Here, the $o$-term can be bounded by a quantity which is 
independent of $v_1, \ldots, v_{d-1} \in \S_p^{n-1}$. As a consequence,
\begin{align*}
 &
\limsup_{h \to 0} 
\frac{|\lvert f^{(d-1)}(x+h) \rvert_{\mathrm{op}(q)} - \lvert f^{(d-1)}(x) \rvert_{\mathrm{op}(q)}|}{|h|_p}\\
 \le \; &\sup \{ f^{(d)}(x)[v_1, \ldots, v_{d-1}, v_d] \colon v_1, \ldots, v_d \in \S_p^{n-1}\}
 = \lvert f^{(d)}(x) \rvert_{\mathrm{op}(q)}.\qedhere
\end{align*}
\end{proof}

\begin{proof}[Proof of Theorem \ref{HOLqRn}]
    By \cite[Th.\ 6.1]{Bob10}, the $\mathrm{LS}_q$-inequality implies the moment (or $L^r$-norm) inequality
    \begin{equation}\label{LSqLp}
    \lVert g \rVert_r^q \le \lVert g \rVert_q^q + \sigma^q \frac{(\frac{r}{q})^{q-1}-1}{q-1} \lVert |\nabla g|_q \rVert_r^q,
    \end{equation}
    valid for any $r \ge q$ and any locally Lipschitz function $g \colon \R^n \to \R$. Applying \eqref{LSqLp} to $g= |f^{(k)}|_{\mathrm{op}(q)}$ together with Lemma \ref{itGradHess} yields
    \[
    \lVert f^{(k)} \rVert_{\mathrm{op}(q),r}^q \le \lVert f^{(k)} \rVert_{\mathrm{op}(q),q}^q + \sigma^q \frac{(\frac{r}{q})^{q-1}-1}{q-1} \lVert f^{(k+1)} \rVert_{\mathrm{op}(q),r}^q
    \]
    for any $k=1, \ldots,d-1$. Altogether, this yields
    \begin{align*}
        \lVert f \rVert_r^q &\le \lVert f \rVert_q^q + \sum_{k=1}^{d-1} \Big(\sigma^q \frac{(\frac{r}{q})^{q-1}-1}{q-1}\Big)^k \lVert f^{(k)} \rVert_{\mathrm{op}(q),q}^q\\
        &\hspace{1.5cm} + \Big(\sigma^q \frac{(\frac{r}{q})^{q-1}-1}{q-1}\Big)^d \lVert f^{(d)} \rVert_{\mathrm{op}(q),\infty}^q\\
        &\le \frac{4\sigma^q}{\log(2)} \lVert f^{(1)} \rVert_{\mathrm{op}(q),q}^q + \sum_{k=1}^{d-1} \Big(\sigma^q \frac{(\frac{r}{q})^{q-1}-1}{q-1}\Big)^k \lVert f^{(k)} \rVert_{\mathrm{op}(q),q}^q\\
        &\hspace{1.5cm} + \Big(\sigma^q \frac{(\frac{r}{q})^{q-1}-1}{q-1}\Big)^d \lVert f^{(d)} \rVert_{\mathrm{op}(q),\infty}^q,
    \end{align*}
    where the second step follows from the $\mathrm{SG}_q$-inequality \eqref{SGq}. Noting that
    \[
    \frac{(\frac{r}{q})^{q-1}-1}{q-1} + \frac{4}{\log(2)} \le \frac{4}{\log(2)q^{q-1}(q-1)}r^{q-1} \le \frac{4}{\log(2)(q-1)^q}r^{q-1},
    \]
    we arrive at the estimate
    \begin{align}\label{iteratedBd}
    \begin{split}
        \lVert f \rVert_r^q &\le \sum_{k=1}^{d-1} \Big(\frac{4}{\log(2)} \Big(\frac{\sigma}{q-1}\Big)^q r^{q-1}\Big)^k \lVert f^{(k)} \rVert_{\mathrm{op}(q),q}^q\\
        &\hspace{0.5cm} + \Big(\frac{4}{\log(2)} \Big(\frac{\sigma}{q-1}\Big)^q r^{q-1}\Big)^d \lVert f^{(d)} \rVert_{\mathrm{op}(q),\infty}^q.
        \end{split}
    \end{align}

    To see (1), plugging the assumptions into \eqref{iteratedBd} leads to
    \begin{align*}
    \lVert f \rVert_r^q &\le \sigma^{qd} \Big(\frac{4}{\log(2) (q-1)^q}\Big)^d \sum_{k=1}^d r^{(q-1)k}\\
    &\le \sigma^{qd} \Big(\frac{4}{\log(2) (q-1)^q}\Big)^d \frac{1}{1-r^{1-q}} r^{(q-1)d}\\
    &\le \sigma^{qd} \Big(\frac{4}{\log(2) (q-1)^q} \cdot \frac{q^{q-1}}{q^{q-1}-1}\Big)^d r^{(q-1)d}
    \end{align*}
    and hence
    \begin{align*}
    \lVert f \rVert_r &\le \Big(\sigma^p \frac{4^{p-1}}{\log^{p-1}(2) (q-1)^p} \cdot \Big(\frac{q^{q-1}}{q^{q-1}-1}\Big)^{p-1}r\Big)^{d/p}\\
    &= \Big(\sigma^p \frac{4^{p-1}p(p-1)^p}{(\log(2)(p^{1/(p-1)}-(p-1)^{1/(p-1)}))^{p-1}} r\Big)^{d/p}
    \end{align*}
    for any $r \ge q$. If $r \le q$, we obtain
    \begin{align*}
    \lVert f \rVert_r \le \lVert f \rVert_q &\le \Big(\sigma^p \frac{4^{p-1}p(p-1)^p}{(\log(2)(p^{1/(p-1)}-(p-1)^{1/(p-1)}))^{p-1}} q\Big)^{d/p}\\
    &= \Big(\sigma^p \frac{4^{p-1}p^2(p-1)^{p-1}}{(\log(2)(p^{1/(p-1)}-(p-1)^{1/(p-1)}))^{p-1}}\Big)^{d/p}.
    \end{align*}
    Hence, for any $m\ge 1$ it holds that
    \[
    \lVert |f|^{p/d} \rVert_m = \lVert f \rVert_{pm/d}^{p/d} \le \gamma m,\qquad \gamma = \sigma^p \frac{4^{p-1}p^2(p-1)^{p-1}}{(\log(2)(p^{1/(p-1)}-(p-1)^{1/(p-1)}))^{p-1}}.
    \]
    Now it follows by Lemma \ref{ConcLemma} (1) applied to $g=|f|^{p/d}$ that $\int \exp(c'|f|^{p/d}) d\mu \le 2$ for $c'=1/(2\gamma e)$, which proves (1).

    To show (2), taking $q$-th roots in \eqref{iteratedBd} leads to
    \begin{align*}
        \lVert f \rVert_r &\le \sum_{k=1}^{d-1} \Big(\frac{4}{\log(2)}\Big)^{k/q} \Big(\frac{\sigma}{q-1}\Big)^k r^{k/p} \lVert f^{(k)} \rVert_{\mathrm{op}(q),L^q}\\
        &\hspace{0.5cm} + \Big(\frac{4}{\log(2)}\Big)^{d/q} \Big(\frac{\sigma}{q-1}\Big)^d r^{d/p} \lVert f^{(d)} \rVert_{\mathrm{op}(q),L^\infty}.
    \end{align*}
    Applying Lemma \ref{ConcLemma} (2) with $L=d$ and $\ell=1$, we may bound $\mu(|f| \ge t)$ by
    \[
    2 \exp\Big\{- \frac{\log(2)C'}{qd^pe^p} \min\Big(\min_{k=1, \ldots,d-1} \Big(\frac{t}{\lVert f^{(k)} \rVert_{\mathrm{op}(q),q}}\Big)^{p/k}, \Big(\frac{t}{\lVert f^{(d)} \rVert_{\mathrm{op}(q),\infty}}\Big)^{p/d}\Big)\Big\},
    \]
    where
    \[
    C'= \frac{\log^{p/q}(2)(q-1)^{p}}{4^{p/q}\sigma^p} = \frac{\log(2)^{p-1}}{4^{p-1}(p-1)^p\sigma^p}.
    \]
    This finishes the proof.
\end{proof}

Let us briefly discuss another subtlety. As we have mentioned in the discussion below \eqref{SGq}, $\mathrm{LS}_q(\sigma^q)$ implies $\mathrm{LS}_2(\sigma^2)$ for $q \ge 2$ (up to constant), so one may think of applying standard log-Sobolev results for measures satisfying $\mathrm{LS}_q$-inequalities (which would lead to weaker tails but some benefits like better tractable norms in the concentration inequalities). However, recall that such implications are only valid if one does not change the modulus of gradient. In particular, starting with an $\mathrm{LS}_q$-inequality on $\R^n$, one has $\Gamma f = |\nabla f|_q$, which will lead to an $\mathrm{LS}_2$-inequality with respect to $|\nabla f|_q$, which is not the same as the standard log-Sobolev inequality.

To conclude this section, let us provide an example. One of the simplest situations in which higher order concentration bounds apply are fluctuations of quadratic forms. The classical result in this context is the famous Hanson--Wright inequality, which states that if $X=(X_1, \ldots,X_n)$ is a random vector in $\R^n$ with independent components $X_i$ such that $\E X_i=0$ and $\lVert X_i \rVert_{\Psi_2} := \inf \{c > 0 \colon \E \exp(X_i^2/c^2) \le 2\} \le K$ and $A$ is an $n \times n$ matrix, then
\[
\P(|X^TAX - \E X^TAX| \ge t) \le 2\exp\Big\{-c\min \Big(\Big(\frac{t}{K^2\lVert A \rVert_\mathrm{HS}}\Big)^2, \frac{t}{K^2\lVert A \rVert_\mathrm{op}}\Big)\Big\}
\]
for any $t \ge 0$. Often, one further assumes $\mathrm{Var}(X_i)=1$, in which case $\E X^TAX = \mathrm{tr}(A)$. Summing up, the Hanson--Wright inequality yields that quadratic forms in independent sub-Gaussian random variables admit two levels of tail decay: for small values of $t$, a sub-Gaussian one which scales with the Hilbert--Schmidt norm of $A$, and for large values of $t$ a subexponential one which scales with the operator norm of $A$.

Various extensions and generalizations of the Hanson--Wright inequality are known. By the higher order concentration results in this note, we can easily add another one which yields concentration bounds for quadratic forms in presence of an $\mathrm{LS}_q$-inequality.

\begin{proposition}\label{HWI}
    Let $\mu$ be a probability measure on $\R^n$ which satisfies an $\mathrm{LS}_q$-inequality as well as $\int x_i d\mu = 0$ for any $i=1, \ldots,n$, and define $f\colon \R^n \to \R$ by $f(x) := x^TAx$, where $A=(a_{ij})_{ij}$ is a symmetric $n \times n$ matrix. Then,
    \[
    \mu\Big(\Big|f - \int f(x)d\mu\Big| \ge t\Big) \le 2\exp\Big\{-c_p\min \Big(\Big(\frac{t}{\sigma^2\lVert A \rVert_{\mathrm{HS}(q)}}\Big)^p, \Big(\frac{t}{\sigma^2\lVert A \rVert_{\mathrm{op}(q)}}\Big)^{p/2}\Big)\Big\}
    \]
    for any $t \ge 0$.
\end{proposition}

Assuming $A$ to be symmetric is a matter of convenience, and the result can easily be extended to non-symmetric matrices $A$.

\begin{proof}[Proof of Proposition \ref{HWI}]
    Note that
    \[
    \partial_i f(x) = 2\sum_{j=1}^n a_{ij}x_j, \qquad \partial_{ij} f(x) = 2a_{ij}.
    \]
    In particular, $f''(x) = 2A$. Moreover, each $\partial_i f$ is centered, so that by the $\mathrm{SG}_q$-inequality \eqref{SGq} (applied to each summand $\partial_i f$),
    \[
    \int \sum_{i=1} ^n |\partial_i f|^q d\mu \le \frac{4\sigma^q}{\log(2)} \int \sum_{i,j=1}^n 2^q|a_{ij}|^q d\mu,
    \]
    thus $\lVert f^{(1)} \rVert_{\mathrm{op}(q),q} \le c_p \sigma \lVert A \rVert_{\mathrm{HS}(q),q}$. Plugging into Theorem \ref{HOLqRn} (2) for $d=2$ directly leads to the result.
\end{proof}

\section{Log-Sobolev-type inequalities and concentration on \texorpdfstring{$\ell_p^n$}{l\_p\^n}-spheres}\label{LSISphSec}

\subsection{Background and related work}\label{RelWork}
For $p \in [1,\infty)$, $n \in \N$, and $x=(x_1,\ldots,x_n)$, let $\lvert x \rvert_p := (\sum_{i=1}^n |x_i|^p)^{1/p}$ the $\ell_p^n$-norm of $x$ and $\S_p^{n-1} := \{ x \in \R^n \colon \lvert x \rvert_p = 1\}$ the unit sphere with respect to $\ell_p^n$. On $\S_p^{n-1}$, there are two classical probability measures. The first one is the surface measure $\nu_{p,n}$, defined as the normalized $(n-1)$-dimensional Hausdorff measure on $\S_p^{n-1}$. The second one is the cone probability measure $\mu_{p,n}$, defined by
\[
\mu_{p,n}(A) := \frac{\vol_n(\{t\theta \colon t \in [0,1], \theta \in A\})}{\vol_n(\mathbb{B}_p^n)}
\]
for any Borel set $A \subset \S_p^{n-1}$, where $\mathbb{B}_p^n$ is the $\ell_p^n$-unit ball in $\R^n$. By definition, $\mu_{p,n}(A)$ is the normalized volume of the intersection of $\mathbb{B}_p^n$ with the cone which intersects $\S_p^{n-1}$ in $A$. For $p=1$ and $p=2$, $\nu_{p,n}$ and $\mu_{p,n}$ agree. See \cite[Sect.\ 9.3.1]{PTT19} for details.

Let us collect some previous results related to concentration of measure on $\ell_p^n$-spheres. The first bounds go back to Gromov and Milman \cite{GM87} in the context of isoperimetric inequalities and concentration results for uniformly convex Banach spaces. In the case of $\ell_n^p$-spheres, their results (as restated in \cite[Th.\ 2.16]{Led01}) yield that if $p > 1$, then
\[
I_{\mu_{p,n}}^{|\cdot|_p}(r) = C_pe^{-c_pnr^{\max\{2,p\}}}
\]
for suitable constants $C_p,c_p>0$. Here, for any metric space $(X,d)$ and any Borel probability measure $\mu$,
\[
I_\mu^d(r) := \sup\{1-\mu(A_r) \colon \mu(A) \ge 1/2\}
\]
with $A_r := \{x \in X \colon d(x,A)<r\}$ is the so-called concentration function. By \cite[Prop.\ 1.7\&1.8]{Led01}, this result may be rewritten as
\begin{equation}\label{Lipschitz1}
    \mu_{p,n}\Big(\Big|f - \int f d\mu_{p,n}\Big| \ge t\Big) \le C_p \exp\{ - c_p n t^{\max\{2,p\}}\}
\end{equation}
for every $1$-Lipschitz function $f \colon \S_p^{n-1} \to \R$ for suitable constants $C_p,c_p>0$.

As pointed out in \cite{Led01}, it is possible to reprove these results by an approach which relates the cone measure to probability measures with densities of the form $e^{-c|x|_p^p}$ on $\R^n$, a method we will discuss in more detail at a later point. In \cite{AV00}, such an approach was used especially to prove analogous results for the (hitherto open) case $p=1$ (more generally, for $1 \le p \le 2$). As shown by Schechtman and Zinn \cite{SZ00}, the same results hold if $\S_p^{n-1}$ is equipped with the Euclidean metric $|\cdot|_2$. A different type of question was addressed by the same authors in \cite{SZ90}, where they proved concentration bounds for the $\ell_{p'}^n$-norm of vector on the $\ell_p^n$-sphere if $1 \le p < p' < \infty$. Note that all these results hold for the cone measure $\mu_{p,n}$ (in this context, remark that the ``normalized surface measure'' from \cite[Eq.\ (2.26)]{Led01} actually equals the cone measure).

The relation between the cone and the surface measure was investigated by Naor and Romik \cite{NR03,Nao07}. In particular, it holds that
\[
\lVert \mu_{p,n} - \nu_{p,n} \rVert_\mathrm{TV} := \sup\big\{ |\mu_{p,n}(A) - \nu_{p,n}(A)| \colon A \in \mathcal{B}(\S_p^{n-1})\big\} = \frac{C_p}{\sqrt{n}},
\]
where $\mathcal{B}(\S_p^{n-1})$ denotes the Borel $\sigma$-algebra on $\S_p^{n-1}$. In its sharpest form as shown in \cite{Nao07}, this result is an application of refined concentration bounds for $\ell_{p'}^n$-norms on the $\ell_p^n$-sphere, continuing the line of research from \cite{SZ90}. Moreover, $\nu_{p,n}$ is absolutely continuous with respect to $\mu_{p,n}$ and vice versa. By means of \cite{Nao07}, results for the cone measure can be transferred to the surface measure. For instance, it holds that
\begin{equation}\label{CSN}
    I_{\sigma_{p,n}}^d(r) \le C_p I_{\mu_{p,n}}^d\Big(\frac{r}{2}\Big)\Big[1 + \frac{1}{\sqrt{n}}\Big|\log I_{\mu_{p,n}}^d\Big(\frac{r}{2}\Big)\Big|^{1-\min\{1/2,1/p\}}\Big]
\end{equation}
for every $r > 0$, where $d$ is any metric which induces the standard topology on $\S_p^{n-1}$.

Finally, as a sort of ``asymptotic'' concentration bounds, let us mention that there are counterparts of the concentration inequalities from \cite{SZ90} in the area of large deviation principles. Corresponding results are due to Kabluchko, Prochno and Th\"{a}le \cite{KPT19}.

\subsection{\texorpdfstring{$\mathrm{LS}_q$}{LS\_q}-inequalities}\label{sec:LSILqSubS}
Our main goal in this section is to prove $\mathrm{LS}_q$-inequalities for the cone measure on $\S_p^{n-1}$. To this end, let us recall a simple proof of \eqref{LSISphere} which can be found in \cite[Ch.\ 9.4]{BCG23}. It makes use of the fact that $\nu_{2,n}$ (or, equivalently, $\mu_{2,n}$) can represented as $Z/\lvert Z \rvert_2$ for a random vector $Z = (Z_1,\ldots,Z_n) \in \R^n$ with i.i.d.\ standard Gaussian (i.\,e., $\mathcal{N}(0,1)$) components. From a different perspective, this means decomposing the $n$-dimensional standard Gaussian distribution into a radial and a directional component which are independent of each other.

Now, one can easily verify that given any smooth function $f$ in a neighbourhood of $\S^{n-1}$, the smooth extension $u(x) := f(x/\lvert x \rvert_2)$ of $f$ satisfies $\nabla f(x) = \lvert x \rvert_2^{-1} \nabla_S f(x/\lvert x \rvert_2)$, where $\nabla_S$ denotes the spherical (or intrinsic) gradient on $f$. Therefore, applying the Gaussian log-Sobolev inequality, we arrive at
\[
\mathrm{Ent}_{\mathcal{N}(0,1)} (u^2) \le 2 \E\lvert Z \rvert_2^{-2} \int_{\S^{n-1}} |\nabla_S f|_2^2 d\mu_{2,n}.
\]
Noting that $\E\lvert Z \rvert_2^{-2} = n-2$, we arrive at \eqref{LSISphere} with a slightly weaker but asymptotically correct constant $1/(n-2)$.

This proof will serve as our guideline for proving $\mathrm{LS}_q$-inequalities for the cone measure on $\ell_p^n$-spheres with arbitrary $p \ge 2$. (Note that in \cite[Ch.\ 5]{BL00}, a similar route has been chosen to establish $\mathrm{LS}_q$-inequalities on $\ell_p^n$-balls.) In view of the Riemannian manifold structure of $\ell_p^n$-spheres, one might wonder whether the Bakry--Emery techniques and their $\Gamma$-calculus extensions might lead to log-Sobolev-type inequalities. However, one may check that the generalized curvature bounds do not apply to $\ell_p^n$-spheres, cf.\ Appendix \ref{sec:Geom}.

First, as in the case of $p=2$ we may renormalize a random vector of i.i.d.\ $p$-generalized Gaussian random variables by its $\ell_p^n$-norm to arrive at the cone measure $\mu_{p,n}$. This fact goes back to \cite{SZ90,RR91}. We quote it as stated in \cite[Th.\ 9.3.2]{PTT19}.

\begin{proposition}\label{p-GaussCone}
    Let $p \ge 1$, and let $Z = (Z_1,\ldots,Z_n)$ be a random vector with i.i.d.\ components $Z_i \sim \mathcal{N}_p$. Then, $Z/\lvert Z \rvert_p \sim \mu_{p,n}$, and $Z/\lvert Z \rvert_p$ is independent of $\lvert Z \rvert_p$.
\end{proposition}

Moreover, as demonstrated in Proposition \ref{LSIp-Gauss}, the $p$-generalized Gaussian distribution satisfies an $\mathrm{LS}_q$-inequality for $p \in [2,\infty)$. Therefore, to derive $\mathrm{LS}_q$-inequalities on $\S_p^{n-1}$, it remains to recall some elementary facts about their manifold structure. For a more detailed presentation of these results see e.\,g.\ \cite[Sect.\ 2\&3]{Sat23}.

For any $p \in (1,\infty)$, $\S_p^{n-1}$ is a (Riemannian) submanifold of $\R^n$ given by the zero set of the function $F_p(x) := |x_1|^p + \ldots + |x_n|^p - 1$. It is of class $\mathcal{C}^\infty$ if $p$ is an even integer, of class $\mathcal{C}^{p-1}$ if $p$ is an odd integer and of class $\mathcal{C}^{\lfloor p \rfloor}$ if $p$ is not an integer.

We have $\nabla F_p(x) = p(\sign(x_1)|x_1|^{p-1}, \ldots, \sign(x_n)|x_n|^{p-1})^T =: px^{p-1}$, where $\mathrm{sign}(x)$ denotes the sign function, and thus, the tangent space in $\theta = \theta_p \in \S_p^{n-1}$ is given by
\[
T_\theta\S_p^{n-1} = \mathrm{ker} F_p'(\theta) = \{ x \in \R^n \colon \langle x, \theta^{p-1} \rangle = 0\}.
\]
On $T_\theta\S_p^{n-1}$, the Riemannian metric is given as the metric induced by the standard Euclidean metric (or inner product).

In particular, let $f$ be a smooth function defined on a neighbourhood of $\S_p^{n-1}$, and let $\theta \in \S_p^{n-1}$. Denote by $P_\theta$ the orthogonal projection onto $T_\theta\S_p^{n-1}$. Then, the intrinsic gradient of $f$ at $\theta$ is given by
\[
\nabla_{\S_p^{n-1}} f(\theta) = \nabla_S f(\theta) = P_\theta \nabla f(\theta),
\]
where $\nabla f(\theta)$ denotes the usual (Euclidean) gradient of $f$ in $\theta$. If $g$ is another smooth function defined on a neighbourhood of $\S_p^{n-1}$ such that $f$ and $g$ coincide on $\S_p^{n-1}$, then $\nabla_S f(\theta) = \nabla_S g(\theta)$ for all $\theta \in \S_p^{n-1}$.

\begin{satz}\label{LSICone}
For any real $p \ge 2$, $\mu_{p,n}$ satisfies an $\mathrm{LS}_q$-inequality with constant of order $n^{-1/(p-1)}$. More precisely, for any smooth enough $f \colon \S_p^{n-1} \to \R$, it holds that
\begin{align*}
\mathrm{Ent}_{\mu_{p,n}}(|f|^q) &\le 4^q q^{q-1} \frac {\Gamma(\frac{n-q}{p})}{p^{q/p}\Gamma(\frac n p)} \int_{\S_p^{n-1}} \lvert \nabla_{S} f(\theta) \rvert_q^q d \mu_{p,n}\\
&\le 3 \cdot 4^q q^{q-1} n^{-\frac{1}{p-1}} \int_{\S_p^{n-1}} \lvert \nabla_S f(\theta) \rvert_q^q d \mu_{p,n},
\end{align*}
where the second inequality is valid for any $n \ge 3$. The same inequalities remain true if we assume $f$ to be defined and smooth in a neighbourhood of $\S_p^{n-1}$ and replace $\nabla_S f$ by the usual (non-intrinsic) gradient $\nabla f$.
\end{satz}

For $p=2$, Theorem \ref{LSICone} recovers the usual logarithmic Sobolev inequality on the sphere with Sobolev constant of order $n^{-1}$. As for the choice of the gradient, while the intrinsic gradient $\nabla_S f$ is the natural gradient of a smooth function $f$ on $\S_p^{n-1}$ and independent of any sort of smooth extension of $f$ onto $\R^n$, working with the usual gradient $\nabla f$ can nevertheless turn out to be convenient. This has two main reasons: first, we will especially be interested in higher order situations in this note, and even for second order derivatives, calculating an intrinsic ``Hessian'' becomes remarkably involved if $p \ne 2$. Moreover, in general it is not even clear whether $|\nabla_S f|_q \le |\nabla f|_q$ (unless $p =2$).

\begin{proof}
Let $f$ be a smooth function on some open neighbourhood of $\S_p^{n-1}$. To define a smooth extension of $f$ onto $\R^n$, let $r := r_p := \lvert x \rvert_p$ for any $x \in \R^n$, $\theta := \theta_p := xr_p^{-1}$ for any $x \ne 0$, and set $u(x) := f(r^{-1} x) = f(\theta)$ for $x \ne 0$, so that $u$ is Lipschitz outside a neighbourhood of the origin. Note that
\[
\frac{\partial}{\partial x_j} (x_1^p + \ldots + x_n^p)^{-1/p} = - \sign(x_j)|x_j|^{p-1} r^{-(p+1)}
\]
and hence,
\[
\frac{\partial}{\partial x_j} (r^{-1} x_i) = \delta_{ij} r^{-1} - x_i\sign(x_j)|x_j|^{p-1}r^{-(p+1)}.
\]
In particular, the function $g(x) := r^{-1} x$ has Jacobian
\[
J_xg = r^{-1}(I - (\theta_i \sign(\theta_j)|\theta_j|^{p-1})_{ij}),
\]
and thus, the gradient of $u$ is given by
\begin{align}
\nabla u(x) &= J_g(x)^T \nabla f(g(x)) = r^{-1}(\nabla f(\theta) - (\sum_j \sign(\theta_i)|\theta_i|^{p-1} \theta_j (\nabla f(\theta))_j)_i)\notag\\
&= r^{-1} (\nabla f(\theta) - \langle \nabla f(\theta), \theta \rangle \theta^{p-1}).\label{Repr1}
\end{align}
For $\theta \in \S_p^{n-1}$, this just reads $\nabla u(\theta) = \nabla f(\theta) - \langle \nabla f(\theta), \theta \rangle \theta^{p-1}$.

Writing $\theta' := \theta^{p-1}$, it follows that
\begin{align}
\lvert \nabla u(\theta) \rvert_q &\le \lvert \nabla f(\theta) \rvert_q + \lvert \langle \nabla f(\theta), \theta \rangle \rvert \lvert \theta' \rvert_q\notag\\
&\le \lvert \nabla f(\theta) \rvert_q + \lvert \nabla f(\theta) \rvert_q \lvert \theta \rvert_p \lvert \theta' \rvert_q \le 2 \lvert \nabla f(\theta) \rvert_q\label{finalEst},
\end{align}
using H\"{o}lder's inequality in the second and the fact that $\lvert \theta \rvert_p = 1 = \lvert \theta' \rvert_q$ in the last step.

To derive an analogue of \eqref{finalEst} for the intrinsic gradient, noting that the orthogonal projection onto $T_{\theta}\S_p^{n-1} = (\theta^{p-1})^\perp$ is given by
\[
P_{\theta}(y) := y - \Big\langle y, \frac{\theta'}{\lvert \theta' \rvert_2} \Big\rangle \frac{\theta'}{\lvert \theta' \rvert_2},
\]
where we have used the notation $\theta' := \theta^{p-1}$ again, we have
\begin{align}
\nabla u(\theta) &= \nabla f(\theta) - \frac{\langle \nabla f(\theta), \theta' \rangle}{\lvert \theta'\rvert_2^2} \theta' + \frac{\langle \nabla f(\theta), \theta' \rangle}{\lvert \theta'\rvert_2^2} \theta' - \langle \nabla f(\theta), \theta \rangle \theta'\notag\\
&= \nabla f(\theta) - \frac{\langle \nabla f(\theta), \theta' \rangle}{\lvert \theta'\rvert_2^2} \theta' + \Big(\frac{\langle \nabla f(\theta), \theta' \rangle}{\lvert \theta'\rvert_2^2} - \langle \nabla f(\theta), \theta \rangle\Big) \theta'\notag\\
&= \nabla_S f(\theta) + \Big\langle \nabla f(\theta), \frac{\theta'}{\lvert \theta' \rvert_2} - \lvert \theta' \rvert_2 \theta \Big\rangle \frac{\theta'}{\lvert \theta' \rvert_2}.\label{Repr2}
\end{align}

Moreover, for $\theta \in \S_p^{n-1}$ we may rewrite
\[
\nabla f(\theta) =  \nabla_S f(\theta) + \Big\langle \nabla f(\theta), \frac{\theta'}{\lvert \theta' \rvert_2}\Big\rangle \frac{\theta'}{\lvert \theta' \rvert_2},
\]
and thus,
\begin{align*}
&\Big\langle \nabla f(\theta), \frac{\theta'}{\lvert \theta' \rvert_2} - \lvert \theta' \rvert_2 \theta \Big\rangle\\
=\, &\Big\langle \nabla_S f(\theta), \frac{\theta'}{\lvert \theta' \rvert_2} - \lvert \theta' \rvert_2 \theta \Big\rangle
   + \Big\langle \Big\langle \nabla f(\theta), \frac{\theta'}{\lvert \theta' \rvert_2}\Big\rangle \frac{\theta'}{\lvert \theta' \rvert_2}, \frac{\theta'}{\lvert \theta' \rvert_2} - \lvert \theta' \rvert_2 \theta \Big\rangle\\
   =\, &- \langle \nabla_S f(\theta), \lvert \theta' \rvert_2 \theta \rangle + \Big\langle \nabla f(\theta), \frac{\theta'}{\lvert \theta' \rvert_2}\Big\rangle (1 - \langle \theta', \theta \rangle)\\
   =\, &- \lvert \theta' \rvert_2 \langle \nabla_S f(\theta), \theta \rangle,
   \end{align*}
using $\nabla_S f(\theta) \in (\theta')^\perp$ in the second and $\langle \theta', \theta \rangle = 1$ in the last step. Together with \eqref{Repr2}, this yields
\[
\nabla u(\theta) = \nabla_S f(\theta) - \langle \nabla_S f(\theta), \theta \rangle \theta'.
\]
From here, we may argue as in \eqref{finalEst} to arrive at
\begin{equation}
\lvert \nabla u(\theta) \rvert_q \le 2 \lvert \nabla_S f(\theta) \rvert_q\label{finalEst2}.
\end{equation}

Now recall that by Lemma \ref{LSIp-Gauss}, $\mathcal{N}_p$ satisfies an $\LSq$-inequality for $q=p/(p-1)$ with constant $\sigma^q = 2 q^{1/p}$. Hence, recalling that $u(x) = f(r^{-1}x)$ we obtain
\begin{align*}
\mathrm{Ent}_{\mu_{p,n}}(|f|^q) &= \mathrm{Ent}_{\mathcal{N}_p^{\otimes n}}(|u|^q) \le 2^q q^{q-1} \int_{\R} \lvert \nabla u \rvert_q^{q} d\mathcal{N}_p\\
&\le 4^q q^{q-1} \E \frac{1}{\lvert Z \rvert_p^{q}} \int_{\S_p^{n-1}} \lvert \nabla_S f(\theta) \rvert_q^q d \mu_{p,n}
\end{align*}
with a random vector $Z=(Z_1,\ldots,Z_n)$ with i.i.d.\ coordinates of distribution $\mathcal{N}_p$. Here, the last step combines \eqref{Repr1} and \eqref{finalEst2} and uses the independence of $Z/\lvert Z \rvert_p$ and $\lvert Z \rvert_p$. The same result is valid for the usual gradient $\nabla f$ if we replace \eqref{finalEst2} by \eqref{finalEst}. Furthermore, by Lemma \ref{NegMompgenG} for $v=q$ it holds that for any $n > q$,
\[
\E  \lvert Z\rvert_p^{-q} = \frac {\Gamma(\frac{n-q}{p})}{p^{q/p} \Gamma(\frac n p)}.
\]

Putting everything together, we thus arrive at
\[
\mathrm{Ent}_{\mu_{p,n}}(|f|^q) \le 4^q q^{q-1} \frac {\Gamma(\frac{n-q}{p})}{p^{q/p}\Gamma(\frac n p)} \int_{\S_p^{n-1}} \lvert \nabla_S f(\theta) \rvert_q^q d \mu_{p,n}
\]
and its analogue for the usual gradient $\nabla f$. Moreover, again by Lemma \ref{NegMompgenG}, we have
\[
\frac{\Gamma(\frac{n-2}{p})}{p^{q/p} \Gamma(\frac n p)} \le (n-q)^{-1} n^{1-q/p} \le 3 n^{-q/p} = 3 n^{-1/(p-1)}
\]
for any $n \ge 3$ and any $p \ge 2$. This finishes the proof.
\end{proof}

In fact, one readily checks that the last estimate of the proof is sharp since it also follows from \cite[Th.\ 2]{Jam13} that a lower bound is given by $(n-q)^{-1/(p-1)}$.

Analogous results for the surface measure $\nu_{p,n}$ appear to be much less immediate, and it is an open question whether $\nu_{p,n}$ also satisfies an $\mathrm{LS}_q$-inequality with a constant of the same order (in $n$). For further comments, we again refer to Appendix \ref{sec:Geom}.

\subsection{Concentration inequalities}
By Theorem \ref{LSICone} we may recover classical concentration inequalities for Lipschitz functions on $\ell_p^n$-spheres of type \eqref{Lipschitz1}. Indeed, if $f \colon \S_p^{n-1} \to \R$ is such that $|\nabla f|_q \le 1$, by \eqref{LipschitzLq} it holds that
\begin{equation}\label{1stOrderCone}
    \mu_{p,n}\Big(\Big|f - \int f d\mu_{p,n}\Big| \ge t\Big) \le 2 \exp\Big\{ - \frac{(q-1)^p}{3^{p-1} \cdot 4^p q} n t^p\Big\}
\end{equation}
for any $t \ge 0$. Recalling that Lipschitz functions are differentiable almost everywhere, it is easy to see that the property $|\nabla f|_q \le 1$ may be replaced by Lipschitz continuity with respect to the $\ell_p^n$-norm.

While the $\mathrm{LS}_q$-inequalities from the previous section are valid for the cone measure $\mu_{p,n}$ only, it is possible to transfer the concentration results to the surface measure $\nu_{p,n}$. Here, a central observation is that by \cite[Lem.\ 2]{NR03}, $\nu_{p,n}$ is absolutely continuous with respect to $\mu_{p,n}$ with density given by
\begin{equation}\label{eq:dens}
h_{n,p}(\theta_p) = \frac{d\nu_{p,n}}{d\mu_{p,n}} (\theta_p) = c_{p,n} \Big(\sum_{j=1}^n |\theta_{p,j}|^{2p-2} \Big)^{1/2} = c_{p,n} \lvert \theta_p^{p-1} \rvert_2
\end{equation}
with normalization constant
\[
c_{p,n}=\Big(\int \Big(\sum_{j=1}^n |\theta_{p,j}|^{2p-2} \Big)^{1/2} d\mu_{p,n}\Big)^{-1}.
\]
In particular, this density is bounded away from zero and infinity. In fact, it is possible to deduce $\mathrm{LS}_q$-inequalities for $\nu_{n,p}$ from this fact, but they depend on the upper and lower bounds on the density, which results in an additional $n$-dependency, leading to non-optimal (concentration) results. By a different approach, we obtain the following first order bound.

\begin{proposition}\label{LipschUnif}
    Let $n \ge 3$. For any $1$-Lipschitz function $f \colon \S_p^{n-1} \to \R$ and any $t \ge 0$, it holds that
    \[
    \nu_{p,n}\Big(\Big|f - \int f d\nu_{p,n}\Big| \ge t\Big) \le 2 (p+1)^{1-\frac{2}{p}} \exp\Big\{ - \frac{(q-1)^p}{3^{p-1} \cdot 8^p q} n t^p\Big\}.
    \]
\end{proposition}

As compared to \eqref{CSN}, we essentially get rid of the logarithmic ``error term''. The core of the proof is the following lemma, which provides an estimate for the $L^2$-norm (with respect to $\mu_{p,n}$) of the density $h_{n,p}$.

\begin{lemma}\label{L2BdDens}
    It holds that
    \[
    \Big( \int_{\S_p^{n-1}} h_{n,p}^2 d\mu_{p,n} \Big)^{1/2} \le (p+1)^{\frac{1}{2} - \frac{1}{p}}.
    \]
\end{lemma}

\begin{proof}
    By definition,
    \[
    \Big( \int h_{n,p}^2 d\mu_{p,n} \Big)^{1/2} = c_{p,n} \Big(\int \sum_{j=1}^n |\theta_{p,j}|^{2p-2} d\mu_{p,n}\Big)^{1/2}.
    \]
    We will deal with the constant $c_{n,p}$ and the integral separately. To bound $c_{n,p}$, recall the elementary inequalities $\lvert x \rvert_{p'} \le \lvert x \rvert_p \le n^{1/p-1/p'} \lvert x \rvert_{p'}$, valid for any $0 < p \le p' < \infty$ and any $x \in \R^n$. Moreover, note that $\lvert x^{p-1} \rvert_2 = \lvert x \rvert_{2p-2}^{p-1}$. Therefore, applying these relations with $p'=2p-2$ and to the power of $p-1$ yields $\lvert \theta^{p-1} \rvert_2 \ge n^{1/p-1/2}$ and hence
    \begin{equation}\label{LowerBdC}
    c_{p,n} \ge n^{1/p-1/2}.
    \end{equation}

    Furthermore, by Lemma \ref{Momente} for $v=2p-2$ we have
    \[
    \int \sum_{j=1}^n |\theta_{p,j}|^{2p-2} d\mu_{p,n} = n \frac{ \Gamma(2-\frac 1 p)\Gamma(\frac{n} {p})} {\Gamma(2+\frac{n-2} {p}) \Gamma(\frac 1 p)},
    \]
    and since $p \le v \le 2p$, the latter is bounded by
    \[
    n \cdot (p+1)^{1 - \frac{2}{p}} n^{2/p-2} = (p+1)^{1 - \frac{2}{p}} n^{2/p-1}
    \]
Taking roots and combining this bound with \eqref{LowerBdC} directly yields to the result.
\end{proof}

\begin{proof}[Proof of Proposition \ref{LipschUnif}]
    Writing $f_\nu := f - \int f d\nu_{p,n}$ and $f_\mu := f - \int f d\mu_{p,n}$, we have
    \begin{align*}
        \int \exp\{\lambda f_\nu(\theta)\} \nu_{p,n}(d\theta) &= \int \exp\Big\{\lambda \int (f_\nu(\theta) - f_\nu(\theta')) \nu_{p,n}(d\theta') \Big\} \nu_{p,n}(d\theta)\\
        &\le \iint \exp\{\lambda (f_\nu(\theta) - f_\nu(\theta'))\} \nu_{p,n}(d\theta') \nu_{p,n}(d\theta)\\
        &= \iint \exp\{\lambda (f_\mu(\theta) - f_\mu(\theta'))\} \nu_{p,n}(d\theta') \nu_{p,n}(d\theta)\\
        &= \int e^{\lambda f_\mu} h_{n,p} d\mu_{p,n} \int e^{\lambda (-f_\mu)} h_{n,p} d\mu_{p,n},
    \end{align*}
    where the second step is due to Jensen's inequality. Setting $\sigma^q := 3 \cdot 4^q q^{q-1} n^{-\frac{1}{p-1}}$, it follows from \cite[Eq.\ (1.0.3)]{BZ05} and Theorem \ref{LSICone} that
    \[
    \int e^{\lambda f_\mu} d\mu_{p,n} \le \exp\Big\{ \frac{\sigma^q}{q^q(q-1)} \lambda^q\Big\}
    \]
    for any $\lambda > 0$, and the same bound also holds for $-f_\mu$. Hence, by Cauchy--Schwarz (applied to both exponential integrals) together with Lemma \ref{L2BdDens} in the second step,
    \begin{align*}
        \int e^{\lambda f_\nu} d\nu_{p,n} &\le \Big(\int e^{2\lambda f_\mu} d\mu_{p,n}\Big)^{1/2} \Big(\int e^{2\lambda (-f_\mu)} d\mu_{p,n}\Big)^{1/2} \int h_{n,p}^2 d\mu_{p,n}\\
        &\le (p+1)^{1 - \frac{2}{p}} \exp\Big\{ \frac{2^q\sigma^q}{q^q(q-1)} \lambda^q\Big\}.
    \end{align*}
    From here, the usual optimization arguments as performed in \cite[Th.\ 1.3]{BZ05} immediately lead to the result.
\end{proof}

The easiest way to obtain higher order concentration results for the cone measure $\mu_{p,n}$ on $\S_p^{n-1}$ is to regard $\mu_{p,n}$ as a measure on $\R^n$ which satisfies an $\mathrm{LS}_q$-inequality with respect to the usual gradient $\nabla f$ on $\R^n$. Combining Theorem \ref{HOLqRn} with Theorem \ref{LSICone} (2) (hence, setting $\sigma^q := 3 \cdot 4^q q^{q-1} n^{-1/(p-1)}$), we obtain the following result.

\begin{satz}\label{HOSph}
Let $n \ge 3$, and let $f$ be a function which is $\mathcal{C}^d$ in some neighbourhood of $\S_p^{n-1}$ and with $\mu_{p,n}$-mean zero.
    \begin{enumerate}
        \item Assuming that $\lVert f^{(k)} \rVert_{\mathrm{op}(q),q} \le (3^{1/q} \cdot 4 q^{1/p} n^{-1/p})^{d-k}$ for all $k=1,\ldots,d-1$ and $\lVert f^{(d)} \rVert_{\mathrm{op}(q),\infty} \le 1$, it holds that
        \[
        \int \exp\Big\{c_{p,d}n |f|^{p/d} \Big\} d\mu_{p,n} \le 2,
        \]
        where a possible choice of the constant $c_{p,d}$ is given by
        \[
        c_{p,d} = \frac{(\log(2)(p^{1/(p-1)}-(p-1)^{1/(p-1)}))^{p-1}}{3^{p-1}2^{4p-1}p^3(p-1)^{p-2}e}.
        \]
        \item For any $t\ge 0$, it holds that
        \begin{align*}
        &\mu_{p,n}(|f| \ge t)\le\\
        &2 \exp\Big\{- C_{p,d}n \min\Big(\min_{k=1, \ldots,d-1} \Big(\frac{t}{\lVert f^{(k)} \rVert_{\mathrm{op}(q),q}}\Big)^{p/k}, \Big(\frac{t}{\lVert f^{(d)} \rVert_{\mathrm{op}(q),\infty}}\Big)^{p/d}\Big)\Big\},
        \end{align*}
    where a possible choice of the constant $C_{p,d}$ is given by
    \[
    C_{p,d} = \frac{\log(2)^p}{3^{p-1}4^{2p-1}p^2(p-1)^{p-1}d^pe^p}.
    \]
    \end{enumerate}
\end{satz}

Again, a simple example where concentration of second order applies are Hanson--Wright type inequalities. In particular, we may apply Proposition \ref{HWI} to the cone measure $\mu_{p,n}$ on $\S_p^{n-1}$. Recalling that here, $\sigma$ is of order $n^{-1/p}$, the corresponding result reads
\[
\mu_{p,n}\Big(\Big|f - \mathrm{tr}(A)m_2\Big| \ge t\Big) \le 2\exp\Big\{-c_p\min \Big(\frac{n^2t^p}{\lVert A \rVert_{\mathrm{HS}(q)}^p}, \frac{nt^{p/2}}{\lVert A \rVert_{\mathrm{op}(q)}^{p/2}}\Big)\Big\}.
\]
for any $t \ge 0$, where
\[
m_2 = \int_{\S_p^{n-1}} \theta_1^2 d\mu_{p,n} = \frac{ \Gamma(\frac{3}{p})\Gamma(\frac{n}{p})} {\Gamma(\frac{n+2}{p}) \Gamma(\frac 1 p)}
\]
by Lemma \ref{Momente}.

Similarly to Proposition \ref{LipschUnif}, we may extend Theorem \ref{HOLqRn} to the surface measure $\nu_{p,n}$ by using Lemma \ref{L2BdDens} and performing a Cauchy--Schwartz argument as in the proof of Proposition \ref{HOLqRn}.

\begin{korollar}\label{KorSurfHO}
    Let $n \ge 3$, and let $f$ be a function which is $\mathcal{C}^d$ in some neighbourhood of $\S_p^{n-1}$ and which has $\mu_{p,n}$-mean zero.
    \begin{enumerate}
        \item Assuming that $\lVert f^{(k)} \rVert_{\mathrm{op}(q),q} \le (3^{1/q} \cdot 4 q^{1/p} n^{-1/p})^{d-k}$ for all $k=1,\ldots,d-1$ (where the $L^q$-norms are taken with respect to $\mu_{p,n}$) and $\lVert f^{(d)} \rVert_{\mathrm{op}(q),\infty} \le 1$, it holds that
        \[
        \int \exp\Big\{c_{p,d}'n |f|^{p/d} \Big\} d\nu_{p,n} \le \sqrt{2} (p+1)^{\frac{1}{2} - \frac{1}{p}}.
        \]
        Here we may set $c_{p,d}' = c_{p,d}/2$ for the constant $c_{p,d}$ from Theorem \ref{HOSph}.
        \item For any $t\ge 0$, it holds that
        \begin{align*}
        &\nu_{p,n}(|f| \ge t)\le \sqrt{2} (p+1)^{\frac{1}{2} - \frac{1}{p}}\\
        &\exp\Big\{- C_{p,d}'n \min\Big(\min_{k=1, \ldots,d-1} \Big(\frac{t}{\lVert f^{(k)} \rVert_{\mathrm{op}(q),q}}\Big)^{p/k}, \Big(\frac{t}{\lVert f^{(d)} \rVert_{\mathrm{op}(q),\infty}}\Big)^{p/d}\Big)\Big\}.
        \end{align*}
    Here we may set $C_{p,d}' = C_{p,d}/2$ for the constant $C_{p,d}$ from Theorem \ref{HOSph}.
    \end{enumerate}
\end{korollar}

\begin{proof}
    Noting that
    \begin{align*}
    \int \exp\Big\{c_{p,d}'n |f|^{p/d} \Big\} d\nu_{p,n} &= \int \exp\Big\{c_{p,d}'n |f|^{p/d} \Big\} h_{n,p} d\mu_{p,n}\\
    &\le \Big(\int \exp\Big\{2c_{p,d}'n |f|^{p/d} \Big\} d\mu_{p,n}\Big)^{1/2} \Big( \int h_{n,p}^2 d\mu_{p,n} \Big)^{1/2},
    \end{align*}
    (1) is clear by combining Theorem \ref{HOSph} and Lemma \ref{L2BdDens}. Writing $\nu_{p,n}(|f| \ge t) = \int \mathbbm{1}_{\{|f| \ge t\}} d\nu_{p,n}$, the same arguments also lead to (2).
\end{proof}

One drawback of Corollary \ref{KorSurfHO} is that assumptions are needed which involve the cone measure. In the case of the boundedness of the $L^q$-norms, it is not quite clear how to avoid this. As for recentering the function around $\int f d\nu_{p,n}$ (instead of the $\mu_{p,n}$-mean), at least for $d \le p$ one may proceed similarly as in the proof of Proposition \ref{LipschUnif}. Indeed, in the first step one rewrites (using the notation from the proof of Proposition \ref{LipschUnif})
\[
|f_\nu|^{p/d} = \Big| \int (f_\nu(\theta) - f_\nu(\theta')) \nu_{p,n}(d\theta')\Big|^{p/d} \le \int |f_\nu(\theta) - f_\nu(\theta')|^{p/d} \nu_{p,n}(d\theta'),
\]
where the second step follows from Jensen's inequality, an argument which only works if $|\cdot|^{p/d}$ is convex. In this sense, it is possible to state a version of Corollary \ref{KorSurfHO} for functions with $\nu_{p,n}$-mean zero if $d \le p$. We omit the details.

Another type of results which we only briefly sketch at this points are refinements of Theorem \ref{HOLqRn} in the spirit of \cite[Th.\ 1.2]{BCG17}. The latter is valid for $p=2$ and $d=2$ and essentially states that the conditions on the function $f$ under considerations may be replaced by analogues for the functions $f_a(x) = f(x) - a|x|_2^2/2$ for any $a \in \mathbb{R}$. In particular, one may replace the Hessian $f''$ by $f'' - aI$, which can lead to improvements or facilitate calculations. Essentially, this is due to the fact that the additional term $a|x|_2^2/2$ is constant on $\S_2^{n-1}$, so that $f$ and $f_a$ have the same mean.

For $p \in \N$ and $d=p$, one may proceed similarly. Indeed, note that for a given $\mathcal{C}^p$-smooth function $f$ in a neighbourhood of $\S_p^{n-^1}$ and any $a \in \R$, the function $f_a(x) := f(x) - a\lvert x \rvert_p^p/p!$ satisfies $f - \int f d\mu_{p,n} = f_a - \int f_a d\mu_{p,n}$ on $\S_p^{n-1}$ since the additional term $a\lvert x \rvert_p^p/p!$ is constant on the $\ell_p^n$-sphere. Therefore, one may replace the conditions $\lVert f^{(k)} \rVert_{\mathrm{op}(q),q} \le (3^{1/q} \cdot 4 q^{1/p} n^{-1/p})^{d-k}$ for all $k=1,\ldots,p-1$ and $\lVert f^{(p)} \rVert_{\mathrm{op}(q),\infty} \le 1$ by its analogues for $f_a$. Again, we skip the details.

\section{Concentration bounds for symmetric functions on the sphere}\label{ExSect}

The central aim of this section is to derive concentration results for a rather general framework of functionals given by sequences of symmetric functions on spheres which is due to \cite{GNU17}. In detail, let $h_n(\theta) = h_n(\theta_1,\ldots,\theta_n)$, $n \ge 1$, be a sequence of real functions on $\R^n$ which satisfy the following conditions:
\begin{align}
    h_{n+1}(\theta_1,\ldots,\theta_j,0,\theta_{j+1},\ldots, \theta_n) = h_n(\theta_1,\ldots, \theta_j, \theta_{j+1},\ldots, \theta_n),\label{Prop1}\\
    \frac{\partial}{\partial \theta_j} h_n(\theta_1,\ldots, \theta_j,\ldots, \theta_n) \Big|_{\theta_j = 0} = 0 \text{ for all $j=1,\ldots,n$},\label{Prop2}\\
    h_n(\theta_{\pi(1)}, \ldots, \theta_{\pi(n)}) = h_n(\theta_1, \ldots, \theta_n) \text{ for all $\pi \in S_n$},\label{Prop3}
\end{align}
where $S_n$ denotes the symmetric group.

To give an example, let $X_j$ be random elements (in an arbitrary space) with common distribution $P$ and $F$ a smooth functional on the space of signed measures. Then, the expected value of the functional $F$ applied to the weighted empirical process of the Dirac measures in $X_1,\ldots,X_n$
\[
h_n(\theta_1,\ldots,\theta_n) := \E F(\theta_1(\delta_{X_1} - P) + \ldots + \theta_n(\delta_{X_n} - P))
\]
satisfies properties \eqref{Prop1}--\eqref{Prop3}.

In \cite{GNU17}, the asymptotic behaviour of such sequences $h_n(\theta)$ has been studied, including a Berry--Esseen type bound and Edgeworth expansions around a limit function $h_\infty(\theta) = h_\infty(|\theta|_2)$, given by
\[
h_\infty(\lambda) := \lim_{k\to \infty} h_k\Big(\frac{\lambda}{\sqrt{k}}, \ldots, \frac{\lambda}{\sqrt{k}}\Big)
\]
for any $\lambda > 0$. In particular, \cite[Th.\ 2.2]{GNU17} adapted to $s=4$ reads as follows. Here, for any vector $\alpha = (\alpha_1, \ldots, \alpha_r)$ of non-negative integers we denote by $D^\alpha$ the partial derivatives $\frac{\partial^{\alpha_1}}{\partial\theta_1^{\alpha_1}}\cdots \frac{\partial^{\alpha_r}}{\partial\theta_r^{\alpha_r}}$.

\begin{proposition}\label{GNU2.2}
Assume that $h_n(\theta_1,\ldots, \theta_n)$, $n \ge 1$, satisfies conditions \eqref{Prop1}--\eqref{Prop3}, and suppose that
\[
|D^\alpha h_n(\theta_1, \ldots, \theta_n)| \le B
\]
for all $\theta_1, \ldots, \theta_n$, where $B$ is some positive constant, $\alpha = (\alpha_1, \ldots, \alpha_r)$, $r \le 4$, and $\alpha_j \ge 2$, $j=1,\ldots, r$, $\sum_{j=1}^r (\alpha_j -2) \le 2$. Furthermore, let $|\theta|_2 = 1$. Then,
\[
h_n(\theta_1, \ldots, \theta_n) = h_\infty + \frac{1}{6} \frac{\partial^3}{\partial \lambda^3} h_\infty(\lambda) \Big|_{\lambda=0} \sum_{j=1}^n\theta_j^3 + R_4,
\]
where the remainder term $R_4$ satisfies
\begin{equation}\label{RemTerm}
|R_4| \le \gamma B \sum_{j=1}^n \theta_j^4
\end{equation}
for a suitable constant $\gamma > 0$.
\end{proposition}

By Proposition \ref{GNU2.2}, the functionals $h_n(\theta)$ can be asymptotically expressed in terms of elementary polynomials, giving rise to an Edgeworth-type expansions. In fact, the functionals $h_n(\theta)$ can even be regarded as a general scheme which reflects common features of many Edgeworth-type asymptotic expansions due to a universal behaviour of distributions based upon many independent asymptotically negligible random variables.

We now choose $\theta = (\theta_1, \ldots, \theta_n) \in \S_2^{n-1}$ at random according to the uniform distribution $\nu_{2,n} = \mu_{2,n}$ and study concentration bounds for $h_n(\theta)$ as a functional on $\S_2^{n-1}$. In view of Proposition \ref{GNU2.2}, it is natural to center around $h_\infty$. One then needs to control the fluctuations of elementary polynomials of the form
\begin{equation}\label{ElemPol}
Q_m(\theta) = \sum_{j=1}^n a_j \theta_j^m.
\end{equation}
In \cite[Ch.\ 10.6\&10.7]{BCG23}, concentration results for $Q_m$ were proven if $p=2$ and $m=3,4$, and in principle, they are already sufficient for our purposes. However, we prefer to keep this paper self-contained, and indeed, it turns out that we may easily derive such inequalities for the polynomials $Q_m(\theta)$ on $\S_p^{n-1}$ equipped with the cone measure from the results from Section \ref{LSISphSec}. This generalizes the aforementioned concentration bounds to any $p \ge 2$ and any order $m > p$.

\begin{proposition}\label{DevElemPolS2}
	Let $n \ge 3$, $p \ge 2$, and let $m > p$ be a natural number. Consider a polynomial $Q_m(\theta)$ as in \eqref{ElemPol} on $\S_p^{n-1}$, and assume that the vector of its coefficients $a=(a_1, \ldots, a_n) \in \R^n$ satisfies $|a|_{p/(p-2)} \le 1$. Then, it holds that
	\[
	\int \exp\Big\{c_{m,p} n\Big|Q_m - \int Q_md\mu_{p,n}\Big|^{p/(m-\lceil p \rceil +1)}\Big\} d\mu_{p,n} \le 2
	\]
	for some constant $c_{m,p} > 0$. Here, $\lceil p \rceil$ is the smallest integer greater or equal $p$.
\end{proposition}

For $p=2$, $p/(p-2)$ has to be interpreted as $\infty$, hence regarding $|a|_{p/(p-2)}$ as the maximum norm of $a$.

\begin{proof}
	We apply Theorem \ref{HOSph} (1) for $d=m-\lceil p \rceil +1$. Note that the tensor of the $k$-th derivatives reads 
	\[
	Q_m^{(k)}(\theta) = (m)_k\mathrm{diag}(a_j \theta_j^{m-k}, j=1, \ldots, n),
	\]
	where $(m)_k$ denotes the falling factorial. In particular, using Lemma \ref{Momente},
	\begin{align*}
		\lVert Q_m^{(k)} \rVert_{\mathrm{op}(q),q} &\le \lVert Q_m^{(k)} \rVert_{\mathrm{HS}(q),q}
		= (m)_k\Big( \sum_{j=1}^n \int |a_j|^q |\theta_j|^{q(m-k)} d\mu_{p,n}\Big)^{1/q}\\
		&= (m)_k\Big(\sum_{j=1}^n |a_j|^q c_{m,p} \frac{1}{n^{q(m-k)/p}}\Big)^{1/q}
		\le c_{m,p} (m)_k n^{-(m-k-(p/q))/p}\\
  &= c_{m,p} (m)_k n^{-(m-k-p+1)/p} \le c_{m,p} (m)_k n^{-(m-k-\lceil p \rceil +1)/p}\\
  &= c_{m,p} (m)_kn^{-(d-k)/p}
	\end{align*}
	for any $k=1, \ldots, d-1=m-\lceil p \rceil$, where we have used that $|a|_{p/(p-2)} \le 1$ implies $|a|_\infty = \max_j|a_j| \le 1$.

    For $k=d=m-\lceil p \rceil +1$, one notes that
    \begin{align*}
	\lVert Q_m^{(d)} \rVert_{\mathrm{op}(q),q} &= m!\sup\Big\{\sum_{i=1}^n a_ix_i^{(1)}\cdots x_i^{(d)} \colon x^{(1)}, \ldots, x^{(d)} \in \S_p^{n-1}\Big\}\\
 &\le m!\sup\Big\{\sum_{i=1}^n a_ix_iy_i \colon x,y \in \S_p^{n-1}\Big\}\\
 &= m!\sup\Big\{\Big(\sum_{i=1}^n |a_i|^q|x_i|^q\Big)^{1/q} \colon x \in \S_p^{n-1}\Big\}.
	\end{align*}
 By H\"{o}lder's inequality with exponents $p/(p-q)$ and $p/q$, it follows that
 \[
     \Big(\sum_{i=1}^n |a_i|^q|x_i|^q\Big)^{1/q} \le \Big(\sum_{i=1}^n |a_i|^{pq/(p-q)}\Big)^{(p-q)/(pq)} \Big(\sum_{i=1}^n |x_i|^p\Big)^{1/p} = |a|_{p/(p-2)},
 \]
 where we have used $x \in \S_p^{n-1}$. In fact, it is not hard to verify that the last inequality is sharp (in the sense that the supremum is attained). These arguments remain true for $p=2$ (and are even simpler in this case).
 
 Altogether, this verifies the conditions of Theorem \ref{HOSph} (1) up to a constant depending on $m$ and $p$, so that the result follows after rescaling.
\end{proof}

In the simplest case $p=2$, $Q_3$ is hence of order $n^{-1}$, which gives back \cite[Cor.\ 10.6.2]{BCG23}, and for $m=4$, $Q_4$ is of order $n^{-3/2}$, which corresponds to \cite[Prop.\ 10.7.2]{BCG23}.

Putting everything together, we now obtain the desired concentration bounds for the functionals $h_n(\theta)$. Note that this is the type of result briefly mentioned at the end of \cite[Ch.\ 10.9]{BCG23} without giving details.

\begin{satz}\label{applsymm}
    Let $n \ge 3$. Assume that $h_n(\theta_1,\ldots, \theta_n)$ satisfies conditions \eqref{Prop1}--\eqref{Prop3}, and suppose that
\[
|D^\alpha h_n(\theta_1, \ldots, \theta_n)| \le B
\]
for all $\theta_1, \ldots, \theta_n$, where $B$ is some positive constant, $\alpha = (\alpha_1, \ldots, \alpha_r)$, $r \le 4$, and $\alpha_j \ge 2$, $j=1,\ldots, r$, $\sum_{j=1}^r (\alpha_j -2) \le 2$. Furthermore, let $|\theta|_2 = 1$, and set
\[
B^\ast := \Big|\frac{\partial^3}{\partial \lambda^3} h_\infty(\lambda) \Big|_{\lambda=0}\Big|.
\]
Then, for any $t \ge 0$,
    \[
    \mu_{2,n}(n|h_n(\theta_1, \ldots, \theta_n) - h_\infty| \ge t) \le 2 \exp \Big(- c \min\Big(\frac{1}{B^\ast}, \frac{1}{\gamma B}\Big) t\Big),
    \]
    where $\gamma$ is the constant from \eqref{RemTerm}.
\end{satz}

\begin{proof}
    Applying Proposition \ref{GNU2.2}, we obtain
    \begin{align*}
    &\mu_{2,n}(n|h_n(\theta_1, \ldots, \theta_n) - h_\infty| \ge t)\\
    \le \, &\mu_{2,n}\Big(n\Big| \frac{1}{6} \frac{\partial^3}{\partial \lambda^3} h_\infty(\lambda) \Big|_{\lambda=0} \sum_{j=1}^n\theta_j^3\Big| \ge t\Big)
    + \mu_{2,n}\Big(n\Big| \gamma B \sum_{j=1}^n \theta_j^4 \Big| \ge t\Big).
    \end{align*}
    By Proposition \ref{DevElemPolS2} for $p=2$ and $m=3$ and Markov's inequality, the first summand on the right-hand side may be bounded by $2\exp\{-c_3t/B^\ast\}$ for a suitable constant $c_3 > 0$.
    
    For the second summand, noting that here we do not center around the mean of the fourth order polynomial, we do not directly apply Proposition \ref{DevElemPolS2} for $m=2$ but slightly modify its proof. To this end, note that we also have the (sub-optimal) bound
    \[
	\int \exp\Big\{c_4 n\Big|Q_4 - \int Q_4d\mu_{2,n}\Big|\Big\} d\mu_{2,n} \le 2.
	\]
    To see this, one performs the same arguments as in the proof of Proposition \ref{DevElemPolS2} for $p=2$ but for $d=m-2$ (hence, $d=2$ in our case). On the other hand, using Lemma \ref{Momente}, the mean of $Q_4$ is bounded by $cn^{-1}$. Hence, by triangle inequality, suitable modifying the constant $c_4$, we arrive at
    \[
	\int \exp\{c_4 nQ_4\} d\mu_{2,n} \le 2.
	\]
    Now we may apply this to $Q_4(\theta) = \gamma B \sum_{j=1}^n \theta_j^4$, and together with Markov's inequality we arrive at the bound
    \[
    \mu_{2,n}\Big(n\Big| \gamma B \sum_{j=1}^n \theta_j^4 \Big| \ge t\Big) \le 2\exp\{-c_4t/(\gamma B)\}.
    \]
    This finishes the proof.
\end{proof}

Alternatively, one may combine Corollary 10.6.8 and Corollary 10.7.1 from \cite{BCG23}. Arguing similarly as in the proof of Theorem \ref{applsymm}, this yields
\[
    \mu_{2,n}((n-1)|h_n(\theta_1, \ldots, \theta_n) - h_\infty| \ge t) \le 5 \exp \Big(- \min\Big(\frac{2}{3B^\ast}, \frac{1}{78\gamma B}\Big) t\Big),
    \]
where $\gamma$ is the constant from \eqref{RemTerm}.

\appendix
\section{Auxiliary Tools}\label{AppAux}

\subsection{Concentration of measure}
For the reader's convenience, let us recall some elementary relations from the theory of measure concentration especially adapted to the framework of this article. All over this section, we consider a generic probability space $(\Omega,\mathcal{A},\mu)$. Moreover, $g$ is always assumed to be a real-valued measurable function on $\Omega$, and $\lVert g \rVert_r := (\int |g|^r d\mu)^{1/r}$ denotes its $L^r$-norm for any $r > 0$.

Controlling the $L^r$-norms of $g$ implies a number of exponential moment and tail inequalities which can be found all over the literature in various settings and reformulations. For part (1) of the following lemma, we lean on the presentation in \cite[Eq.\ (2.17)]{BGS19}, while (2) is a slight generalization of parts of the proof of \cite[Th.\ 3.6]{SS20}.

\begin{lemma}\label{ConcLemma}
    \begin{enumerate}
        \item Assuming that $\lVert g \rVert_k \le \gamma k$ for any $k \in \N$ and some constant $\gamma > 0$, it holds that
        \[
        \int \exp\Big\{ \frac{|g|}{2\gamma e}\Big\} d\mu \le 2.
        \]
        \item Assume that there are constants $C_1, \ldots, C_d \ge 0$ such that for any $r \ge q$
  \[
    \lVert g \rVert_r \le \sum_{k = 1}^{d} (C_k r)^{k/p}.
  \]
  Let $L := |\{ k : C_k > 0\}|$ and $\ell := \min \{ k \in \{1,\ldots,d\} : C_k > 0 \}$. Then, we have
        \[
        \mu(|g| \ge t) \le 2 \exp\Big\{-\frac{\log(2)}{q(Le)^{p/\ell}} \min_{k=1,\ldots,d} \frac{t^{p/k}}{C_k}\Big\}
        \]
        for any $t \ge 0$.
    \end{enumerate}
\end{lemma}

\begin{proof}
    To see (1), note that  setting $c = 1/(2 \gamma e)$ and using $k! \ge (\frac{k}{e})^k$, we have
\[
\int \exp(c|g|) d\mu = 1 + \sum_{k=1}^{\infty} c^k \frac{\int |g|^k d\mu}{k!} 
\le 1 + \sum_{k=1}^{\infty} (c \gamma)^k \frac{k^k}{k!}
\le 1 + \sum_{k=1}^{\infty} (c \gamma e)^k = 2.
\]

To see (2), first note that by Chebyshev's inequality we have for any $r \ge 1$
  \begin{equation}  \label{eqn:ChebyshevWithLp}
    \mu (|g| \ge e \lVert g \rVert_r) \le \exp(-r).
  \end{equation}
  Now consider the function
  \[
    \eta_g(t) := \min \left\{ \frac{t^{p/k}}{C_k} \colon k = 1, \ldots, d \right\},
  \]
  with $x/0$ being understood as $\infty$. If we assume $\eta_g(t) \ge q$, we can estimate $e \lVert g \rVert_{\eta_g(t)} \le e \sum_{k = 1}^{d} \mathbbm{1}_{\{C_k \neq 0\}} t = Let$, so that an application of equation \eqref{eqn:ChebyshevWithLp} to $r = \eta_g(t)$ yields
  \[
    \mu(|g| \ge (Le)t) \le \mu(|g| \ge e\lVert g \rVert_{\eta_g(t)}) \le \exp\left( - \eta_g(t) \right).
  \]
  Combining it with the trivial estimate $\mu(\cdot) \le 1$ (in the case $\eta_f(t) \le q$) gives
  \[
    \mu(|g| \ge (Le)t) \le e^q \exp(-\eta_g(t)).
  \]
  
  To remove the $Le$ factor, rescaling the function by $Le$ and using the estimate $\eta_{(Le)g}(t) \ge \eta_g(t)/(Le)^{p/\ell}$ yields
  \[
    \mu(|g| \ge t) \le e^q \exp \left( - \frac{1}{(Le)^{p/\ell}} \eta_g(t) \right).
  \]
  In a last step, note that
  \[
  e^q\exp(- (Le)^{-p/\ell} \eta_g(t)) \le 2 \exp( - \log(2) (Le)^{-p/\ell} \eta(t)/q)
  \]
  in the nontrivial regime $(Le)^{-2/\ell} \eta_g(t) \ge q$, so that
  \[
    \mu(|g| \ge t) \le 2 \exp \left( - \frac{\log(2)}{q(Le)^{p/\ell}} \eta_g(t) \right)
  \]
  as claimed in (2).
\end{proof}

\subsection{Calculations of moments}
In this section, we explicitly calculate and give qualitative estimates of the (absolute) moments of several random elements which appear in the course of this note. First, we consider negative moments of the $\ell_p^n$-norm of an $n$-dimensional vector which has generalized Gaussian distribution $\mathcal{N}_p$.

\begin{lemma}\label{NegMompgenG}
    Let $Z=(Z_1,\ldots,Z_n)$ be a random vector with i.i.d.\ coordinates of distribution $\mathcal{N}_p$, and let $v > 0$. For any $n > v$, it holds that
    \[
    \E  \lvert Z\rvert_p^{-v} = \frac {\Gamma(\frac{n-v}{p})}{p^{v/p} \Gamma(\frac n p)}.
    \]
    Moreover, if $v \le p$, this expression is bounded by $(n-v)^{-1}n^{1-v/p}$.
\end{lemma}

It is no problem to derive similar bounds for $v \ge p$ as well, however we only need the case $v=q$ in this note anyway.

\begin{proof}
Note that for real numbers $p >0, z>0, s>0$, we have by definition of the Gamma and Beta functions 
\begin{align}
\int_0^{\infty} \exp[-\lambda z ] \lambda^{s}  \frac{d\lambda} {\lambda} &= z^{-s} \Gamma(s), \label{gamma}\quad \text{and}\\
\int_0^{\infty}   (p \lambda +  1)^{-z} \lambda^{s}\frac {d\lambda}{\lambda} &= p^{-s}\frac{\Gamma(z-s) \Gamma(s)}{\Gamma(z)}, \label{beta} \text{ provided that }  z > s.
\end{align}
    Hence, it follows that
\begin{align*}
 \Gamma(v/p) \E  \lvert Z\rvert_p^{-v} &= \E \int_0^{\infty} \exp\{-\lambda (|Z_1|^p + \ldots + |Z_n|^p)\} \lambda^{v/p}\frac {d\lambda} {\lambda}\\
 &=  \int_0^{\infty} \Big(\E \exp\{-\lambda |Z_1|^p\} \Big)^n  \lambda^{v/p}\frac {d\lambda}{\lambda}\\
&=  \int_0^{\infty}  \Big(\int c_p\exp\{-(\lambda p +1)\frac{|x|^p} {p} \} dx \Big)^n \lambda^{v/p} \frac {d\lambda}{\lambda}\\
&= \int_0^{\infty}   (p \lambda +  1)^{-n/p} \lambda^{v/p}\frac {d\lambda}{\lambda} = \frac {\Gamma(v/p) \Gamma(\frac{n-v}{p})}{p^{v/p} \Gamma(\frac n p)}.
\end{align*}
Moreover, applying \cite[Th.\ 2]{Jam13} for $x= (n-v)/p$ and $y = v/p \le 1$ yields
\[
\frac{\Gamma(\frac{n}{p})}{\Gamma(\frac{n-v}{p})} \ge \frac{n-v}{p} \Big(\frac{n}{p}\Big)^{\frac{q}{p}-1}
\]
and hence the result.
\end{proof}

Slightly extending these arguments, we can also control the (absolute) moments of an arbitrary entry of $\theta \in \S_p^{n-1}$, say $\theta_1$, under the cone measure $\mu_{p,n}$.

\begin{lemma}\label{Momente}
	For any $v \ge 0$, we have
	\[
	\int_{\S_p^{n-1}} |\theta_1|^v d\mu_{p,n} = \frac{ \Gamma(\frac{1 + v}{p})} {\Gamma(\frac 1 p)} \cdot \frac{\Gamma(\frac{n} {p})} {\Gamma(\frac{v+n}{p})}
 \le \begin{cases}
     n^{-1}(n+v)^{1-v/p}, & 0\le v \le p,\\
     (p+1)^{v/p-1}n^{-v/p}, & p \le v \le 2p,\\
     (1+v)^{v/p-1} n^{-1}(n+p)^{1-v/p}, & v \ge 2p.
 \end{cases}
	\]
 In particular, if $v \in [v_0, v_1]$ for $0 < v_0 \le v_1 < \infty$, then
 \[
 \int_{\S_p^{n-1}} |\theta_1|^v d\mu_{p,n} \le c_{p,v_0,v_1} n^{-v/p}.
 \]
	Moreover, if $v \in \N$ is an odd number, we have
 \[
 \int_{\S_p^{n-1}} \theta_1^v d\mu_{p,n} = 0.
 \]
\end{lemma}

\begin{proof}
	Let $Z=(Z_1, \ldots, Z_n)$ be a vector of of i.i.d.\ random variables with distribution $\mathcal{N}_p$. Then, we have
	\begin{align}
		\int |\theta_1|^v d\mu_{p,n} &= \E \Big[|Z_1|^v\Big(\sum_{j=1}^n |Z_j|^p\Big)^{-v/p}\Big]
		= \E \Big[|Z_1|^v\Big(\sum_{j=1}^n |Z_j|^p\Big)^{-v/p}\Big] \notag \\
		&= \frac{1}{\Gamma(v/p)}  \int_0^{\infty} \E |Z_1|^v   \exp\Big[ - \lambda(|Z_1|^p + \sum_{j=2}^n |Z_j|^p)\Big] \lambda^{v/p} \frac{d\;\lambda}{\lambda} \label{appgamma+} ,
	\end{align}
	using \eqref{gamma} with $s=v/p$. Since $Z_j$ has distribution $\mathcal{N}_p$ we have for $v\ge 0$ by change of variables
	\begin{align}
		\E |Z_1|^v \exp[ -\lambda |Z_1|^p] &= \int |z|^v \exp[ - (1+\lambda\;p) \frac{|z|^p} {p} ] c_p dz \notag \\
		&= (1+\lambda\;p)^{- \frac{v+1}{p}} \int |z|^v \exp[ -  \frac{|z|^p} {p} ]c_p dz  \notag \\
		&= (1+\lambda\;p)^{- \frac{v+1}{p}} p^{\frac{v}{p}}  \frac{\Gamma(\frac{1 + v}{p})}{\Gamma(1/p)}.\label{moment+}
	\end{align}
	
	Thus, by independence of $Z_1, \ldots, Z_n$, \eqref{appgamma+} factorizes into one factor \eqref{moment+} and $n-1$ factors \eqref{moment+} with $v=0$. Hence, \eqref{appgamma+} turns into
	\begin{align}\label{PrelimEst+}
		\frac{1}{\Gamma(v/p)} \frac{\Gamma(\frac{1 + v}{p})}{\Gamma(1/p)} p^{\frac{v}{p}} \int_{0}^{\infty}  (1+\lambda\;p)^{- \frac{v+n}{p}} \lambda^{v/p} \frac{d\;\lambda}{\lambda}.
	\end{align}
	The integral on the right hand side may be evaluated by \eqref{beta} with $z=\frac{v+n}{p}$ and $s= v/p$, leading to
	\begin{equation}\int_{0}^{\infty}  (1+\lambda\;p)^{- \frac{v+n}{p}} \lambda^{v/p} \frac{d\;\lambda}{\lambda}= p^{-v/p} \frac{\Gamma(\frac{n}{p})\Gamma(\frac{v}{p})}{\Gamma(\frac{v+n}{p})}. \label{appbeta+}
	\end{equation}
	Note that here we used $z-s= \frac{n}{p}>0$. Summarizing we conclude by combining \eqref{PrelimEst+} and \eqref{appbeta+}
	\[
		\int |\theta_1|^v d\mu_{p,n} = \frac{ \Gamma(\frac{1 + v}{p})\Gamma(\frac{n}{p})} {\Gamma(\frac{v+n}{p}) \Gamma(\frac 1 p)}\\
		= \frac{ \Gamma(\frac{1 + v}{p})} {\Gamma(\frac 1 p)} \cdot \frac{\Gamma(\frac{n} {p})} {\Gamma(\frac{v+n}{p})}.
	\]
	
    Moreover, applying \cite[Th.\ 2]{Jam13} on the two Gamma function ratios, each time with $y := \frac{v}{p}$, leads to
    \[
    \frac{ \Gamma(\frac{1 + v}{p})} {\Gamma(\frac 1 p)} \cdot \frac{\Gamma(\frac{n} {p})} {\Gamma(\frac{v+n}{p})} \le \begin{cases} (\frac{1}{p})^{v/p} \cdot (\frac{n}{p}(\frac{n+v}{p})^{v/p-1})^{-1}, & 0\le v\le p,\\
    \frac{1}{p}(\frac{1}{p}+1)^{v/p-1} \cdot (\frac{n}{p})^{-v/p}, & p \le v \le 2p,\\
    \frac{1}{p}(\frac{1+v}{p})^{v/p-1} \cdot (\frac{n}{p}(\frac{n}{p}+1)^{v/p-1})^{-1}, & v \ge 2p,
    \end{cases}
    \]
    These expressions can be simplified as claimed in the lemma.

    Finally, the last assertion of the lemma simply follows by symmetry, i.\,e., the fact that the distribution of $\theta_1$ under $\mu_{p,n}$ is invariant under multiplication with $-1$.
\end{proof}

Note that \cite[Th.\ 2]{Jam13} also yields lower bounds which essentially imply sharpness of the estimates given above.

\subsection{Matrix calculus}\label{sec:MatCal}

Let $f \colon \R^{m\times \ell} \to \R$ be a differentiable function (for instance, we may require $f$ to have continuous partial derivatives in any entry $X_{ij}$ of the underlying matrix $X$). Then, we may define the derivative of $f$ with respect to the matrix $X$ as the $m \times \ell$-matrix $\nabla f(X) \equiv \partial f(X)/(\partial X)$ with entries
\[
(\nabla f(X))_{ij} = \Big(\frac{\partial f(X)}{\partial X}\Big)_{ij} = \frac{\partial f(X)}{\partial X_{ij}}.
\]
We can also adapt these definitions to the case where $m=\ell=n$ and we consider symmetric matrices only: indeed, in this case one takes partial derivatives with respect to $X_{ij}$ for any $i \le j$ and defines $\nabla f(X)$ as the corresponding symmetric $n \times n$ matrix.

Similarly, we may consider the situation where $M \colon \R \to \R^{m \times \ell}$ is a matrix-valued differentiable function (e.\,g., we may assume that every component is differentiable). In this case, we define the derivative of $g$ with respect to $t$ to be the $m \times \ell$ matrix $dM(t)/(dt)$ with entries
\[
\Big(\frac{\partial M(t)}{\partial t}\Big)_{ij} = \frac{\partial M_{ij}(t)}{\partial t}.
\]

Let us state some results for elementary types of matrix functionals, where $m=\ell=n$ and the matrices are assumed to be invertible.

\begin{lemma}\label{ElemDer}
Let $M \in \R^{n \times n}$ be an invertible matrix which depends on some real parameter $t$.
    \begin{enumerate}
        \item We have
        \[
        \frac{\partial M^{-1}}{\partial t} = -M^{-1} \frac{\partial M}{\partial t} M^{-1}.
        \]
        \item We have
        \[
        \frac{\partial M^2}{\partial t} = \frac{\partial M}{\partial t} M + M \frac{\partial M}{\partial t}.
        \]
        \item We have
        \[
        \frac{\partial M^{-2}}{\partial t} = -M^{-1} \frac{\partial M}{\partial t} M^{-2} - M^{-2} \frac{\partial M}{\partial t} M^{-1}.
        \]
    \end{enumerate}
\end{lemma}

\begin{proof}
    (1) and (2) are immediate, and (3) follows from combining both.
\end{proof}

For any differentiable functions $f \colon \R^{m \times \ell} \to \R$ and $M \colon \R \to \R^{m \times \ell}$, $f(M(t))$ is a real-valued function of a real parameter whose derivatives are given by
\begin{equation}\label{MatrixChain}
	\frac{\partial g(t)}{\partial t} = \Big\langle \nabla f(g(t)), \frac{\partial g(t)}{\partial t} \Big\rangle_\mathrm{HS}.
\end{equation}
Indeed, this follows from the matrix calculus chain rule as stated in \cite[Lem.\ 2.2 (2)]{GS23}.

\subsection{Geometric aspects}\label{sec:Geom}

Recall that if $(M,g)$ is a compact connected smooth Riemannian manifold with Riemannian metric $g$ equipped with the normalized Riemannian volume element $d\nu$, a logarithmic Sobolev inequality holds once there exists a strictly positive lower bound on the Ricci curvature tensor, i.\,e., $c(M) > 0$, where $c(M)$ is the infimum of the Ricci curvature tensor over all unit tangent vectors. For $\S_p^{n-1}$, this condition is violated however.

\begin{proposition}\label{prop:Ric}
If $p > 2$, it holds that $c(\S_p^{n-1})=0$.
\end{proposition}

\begin{proof}
Assume $n\ge 3$. Adapting the notation from \cite{ALV23}, we note that $\S_p^{n-1}$ can be regarded as a so-called separable hyperplane $H$ given by $ \sum_{j=1}^n f_i(x_i) = 1$, where in our case $f_i (x_i) = |x_i|^p$, $i=1,\ldots, n$. Let $x:=(x_1, \ldots, x_n)\in \S_p^{n-1}$ with $ x_n >0 $ and, for simplicity of notation, also $x_i \ge 0$ for all $i = 1, \ldots,n-1$. Then,
\[
\Big(x_1, \ldots, x_{n-1}, x_n:=\Big( 1- \sum_{j=1}^{n-1} x_j^p \Big)^{1/p}\Big)
\]
is a local chart around $x$. A (non-orthogonal) basis of the tangent space at $x$ is given by
\[
X_j:=e_j- \frac{x_i^{p-1}}{x_n^{p-1} } e_n,\quad j=1, \ldots, n-1,
\]
for the standard basis $e_1,\ldots,e_n$ of $\R^n$. By \cite[Eq.\ (4)]{ALV23}, the sectional curvature at $x$ in the plane $(X_i,X_j)$ with $1\le i< j\le n-1$ is then given by
\[
K_x(X_i, X_j) = \frac{ (p-1)^2 (x_i x_j x_n)^{p-2} (x_i^p + x_j^p + x_n^p)} {(x_i^{2 q} + x_j^{2 q} + 
   x_n^{2 q}) (\sum_{j=1}^n x_j^{2q})},
   \]
where $q=p-1$. Hence the sectional curvature $K_x(X_i, X_j)$ vanishes if $x_i=0$ or $x_j=0$.

Now let $x \in \S_p^{n-1}$ be a point with $x_n \ne 0$ and $x_i=0$ for some $1\le i \le n-1$. For an orthonormal system of tangent vectors in $x$ like $Z_1:=X_i, Z_2,\ldots, Z_{n-1}$, say, we may express the Ricci curvature at $x$ in the direction $X_i$ in terms of the sectional curvature as
\[ \mathrm{Ric}_x(X_i) = \frac{1}{n - 1}\sum_{j = 1}^{n - 1} K_x(X_i, Z_j). \]
Since $K_x(X_i, Z_j)$ is linear in $Z_j$ which may be written as linear combination of $X_j$, $j \ne i$, we have at $x$ with $x_i=0$ 
\[  \mathrm{Ric}_x(X_i) =0.\qedhere \]
\end{proof}

Similar remarks also hold for the cone measure, which may be expressed as
\[
\frac {d\mu_{n,p}}{d\nu_{n,p}}= c_{p,n}^{-1} \exp\{ -\Phi\}, \qquad \text{where } \Phi= \frac 1 2\log\Big( \sum_{j=1}^n x_j^{2p-2}\Big)
\]
(cf.\ \eqref{eq:dens}). Indeed, observe that at points $x\in \ell_p^n$ with $x_i=0$ for some $i$, we have $(\mathrm{Hess}\;\Phi+ K_x)(X_i,X_i) =0$, which by the arguments from the proof of Proposition \ref{prop:Ric} violates the extended Bakry--Emery curvature lower bounds by $c\; g(X_i,X_i)$ (using $\Gamma$-calculus, see e.\,g.\ \cite[Cor.\ 4.4.24]{AGZ10}) as well.

\end{document}